\documentclass[english,british]{article}
\usepackage[T1]{fontenc}
\usepackage[latin9]{inputenc}
\usepackage{geometry}
\usepackage{array}
\usepackage{float}
\usepackage{enumitem}
\usepackage{amsmath}
\usepackage{amsthm}
\usepackage{amssymb}
\usepackage[all]{xy}

\makeatletter

\providecommand{\tabularnewline}{\\}

  \theoremstyle{plain}
  \newtheorem*{thm*}{\protect\theoremname}
  \theoremstyle{remark}
  \newtheorem*{rem*}{\protect\remarkname}
  \theoremstyle{remark}
  \newtheorem*{note*}{\protect\notename}
  \theoremstyle{definition}
  \newtheorem*{example*}{\protect\examplename}
\theoremstyle{plain}
\newtheorem{thm}{\protect\theoremname}[section]
  \theoremstyle{definition}
  \newtheorem{defn}[thm]{\protect\definitionname}
  \theoremstyle{remark}
  \newtheorem*{notation*}{\protect\notationname}
  \theoremstyle{plain}
  \newtheorem{prop}[thm]{\protect\propositionname}
  \theoremstyle{plain}
  \newtheorem{lem}[thm]{\protect\lemmaname}
  \theoremstyle{plain}
  \newtheorem{conjecture}[thm]{\protect\conjecturename}

\newcommand{\xyR}[1]{
 \xydef@\xymatrixrowsep@{#1}}
\newcommand{\xyC}[1]{
  \xydef@\xymatrixcolsep@{#1}}
\date{}
\usepackage{hyperref}
    \DeclareFontFamily{U}{wncy}{}
    \DeclareFontShape{U}{wncy}{m}{n}{<->wncyr10}{}
    \DeclareSymbolFont{mcy}{U}{wncy}{m}{n}
    \DeclareMathSymbol{\Sh}{\mathord}{mcy}{"58} 
\usepackage{multicol}

\makeatother

\usepackage{babel}
  \addto\captionsbritish{\renewcommand{\conjecturename}{Conjecture}}
  \addto\captionsbritish{\renewcommand{\definitionname}{Definition}}
  \addto\captionsbritish{\renewcommand{\examplename}{Example}}
  \addto\captionsbritish{\renewcommand{\lemmaname}{Lemma}}
  \addto\captionsbritish{\renewcommand{\notationname}{Notation}}
  \addto\captionsbritish{\renewcommand{\notename}{Note}}
  \addto\captionsbritish{\renewcommand{\propositionname}{Proposition}}
  \addto\captionsbritish{\renewcommand{\remarkname}{Remark}}
  \addto\captionsbritish{\renewcommand{\theoremname}{Theorem}}
  \addto\captionsenglish{\renewcommand{\conjecturename}{Conjecture}}
  \addto\captionsenglish{\renewcommand{\definitionname}{Definition}}
  \addto\captionsenglish{\renewcommand{\examplename}{Example}}
  \addto\captionsenglish{\renewcommand{\lemmaname}{Lemma}}
  \addto\captionsenglish{\renewcommand{\notationname}{Notation}}
  \addto\captionsenglish{\renewcommand{\notename}{Note}}
  \addto\captionsenglish{\renewcommand{\propositionname}{Proposition}}
  \addto\captionsenglish{\renewcommand{\remarkname}{Remark}}
  \addto\captionsenglish{\renewcommand{\theoremname}{Theorem}}
  \providecommand{\conjecturename}{Conjecture}
  \providecommand{\definitionname}{Definition}
  \providecommand{\examplename}{Example}
  \providecommand{\lemmaname}{Lemma}
  \providecommand{\notationname}{Notation}
  \providecommand{\notename}{Note}
  \providecommand{\propositionname}{Proposition}
  \providecommand{\remarkname}{Remark}
  \providecommand{\theoremname}{Theorem}
\providecommand{\theoremname}{Theorem}

\begin{document}
\selectlanguage{english}%
\global\long\def\zz{\mathbb{Z}}
\global\long\def\im{\operatorname{im}}
\global\long\def\re{\operatorname{re}}
\global\long\def\rr{\mathbb{R}}
\global\long\def\cc{\mathbb{C}}
\global\long\def\vv{\mathbb{V}}
\global\long\def\SL{\operatorname{SL}}
\global\long\def\ord{\operatorname{ord}}
\global\long\def\Spec{\operatorname{Spec}}
\global\long\def\qq{\mathbb{Q}}
\global\long\def\pp{\mathbb{P}}
\global\long\def\Hom{\operatorname{Hom}}
\global\long\def\id{\operatorname{id}}
\global\long\def\gcd{\operatorname{gcd}}
\global\long\def\lcm{\operatorname{lcm}}
\global\long\def\Sym{\operatorname{Sym}}
\global\long\def\ff{\mathbb{F}}
\global\long\def\nn{\mathbb{N}}
\global\long\def\hh{\mathbb{H}}
\global\long\def\Ann{\operatorname{Ann}}
\global\long\def\End{\operatorname{End}}
\global\long\def\sgn{\operatorname{sgn}}
 \global\long\def\rk{\operatorname{rk}}
\global\long\def\Br{\operatorname{Br}}
 \global\long\def\Hasse{\operatorname{Hasse}}
\global\long\def\GL{\operatorname{GL}}
\global\long\def\Coll{\operatorname{Col}}
 \global\long\def\Tr{\operatorname{Tr}}
\global\long\def\aa{\mathbb{A}}
\global\long\def\ld{\operatorname{in}_{<}}
\global\long\def\ob{\operatorname{Ob}}
\global\long\def\mor{\operatorname{mor}}
\global\long\def\ext{\operatorname{Ext}}
\global\long\def\tor{\operatorname{Tor}}
\global\long\def\cok{\operatorname{coker}}
\global\long\def\ilim{\varprojlim}
\global\long\def\Gal{\operatorname{Gal}}

\selectlanguage{british}%

\title{The Picard Group of Various Families of $(\zz/2\zz)^{4}$-invariant
Quartic K3 Surfaces}

\author{Florian Bouyer}
\maketitle
\begin{abstract}
The subject of this paper is the study of various families of quartic
K3 surfaces which are invariant under a certain $(\zz/2\zz)^{4}$
action. In particular, we describe families whose general member contains
$8,16,24$ or $32$ lines as well as the $320$ conics found by Eklund
\cite{eklund2010} (some of which degenerate into the mentioned lines).
The second half of this paper is dedicated to finding the Picard group
of a general member of each of these families, and describing it as
a lattice. It turns out that for each family the Picard group of a
very general surface is generated by the lines and conics lying on
said surface.
\end{abstract}

\section{Introduction}

Consider the $(\zz/2\zz)^{4}$ subgroup of $\mathrm{Aut}(\pp_{\overline{\qq}}^{3})$
generated by the four transformations
\[
[x:y:z:w]\mapsto[y:x:w:z],[z:w:x:y],[x:y:-z:-w],[x:-y:z:-w].
\]
In this paper we look at various families of quartic surfaces in $\pp_{\overline{\qq}}^{3}$
which are invariant under these transformations. The family of all
such quartics is known to be parameterised by $\pp^{4}$ and has been
studied extensively by \cite{BarthNieto,eklund2010}. Barth and Nieto,
\cite{BarthNieto}, studied the moduli space of invariant quartics
which contain a line, while Eklund \cite{eklund2010} looked at those
that contain a conic. It turns out that a very general invariant quartic
contains at least $320$ conics and Eklund uses that, among other
tools, to calculate the Picard group of a very general invariant quartic.
Barth and Nieto the locus of invariant quartics containing lines to
be a quintic threefold $N_{5}\subset\pp^{4}$, plus the tangent cones
of the $10$ singular points of $N_{5}$. Eklund studied the quartic
surfaces parameterised by $N_{5}$, so in this paper we look at the
surfaces parameterised by the cones. 

In particular we consider :
\begin{itemize}
\item a four dimensional family $\mathcal{X}$,
\item a three dimensional family $\mathcal{X}_{C,D,E}$,
\item a two dimensional family $\mathcal{X}_{C,D}$,
\item a one dimensional family $\mathcal{X}_{B}$,
\item a one dimensional family $\mathcal{X}_{C}$,
\item a specific quartic K3 surface $Y$,
\item and the Fermat quartic, $F_{4}$.
\end{itemize}
For each of these families, we look at the lines a very general member
contains. We use these and the $320$ conics that Eklund found to
calculate the Picard group of a very general member. Our main result
(Theorem \ref{thm:Main Result Final}) can be summarised as follows:
\begin{thm*}[Summarised Theorem \ref{thm:Main Result Final}]
~

\begin{itemize}
\item A very general member of $\mathcal{X}$ contains no lines, and has
Picard rank $16$,
\selectlanguage{english}%
\item A very general member of $\mathcal{X}_{C,D,E}$ contains exactly $8$
lines, and has Picard rank $17$,
\item A very general member of $\mathcal{X}_{C,D}$ contains exactly $16$
lines, and has Picard rank $18$,
\selectlanguage{british}%
\item A very general member of $\mathcal{X}_{B}$ contains exactly $24$
lines, and has Picard rank $19$,
\item A very general member of $\mathcal{X}_{C}$ contains exactly $32$
lines, and has rank $19$,
\item The surface $Y$ contains exactly $32$ lines, and has Picard rank
$20$,
\item The Fermat quartic, $F_{4}$, contains exactly $48$ lines, and has
Picard rank $20$.
\end{itemize}
Possibly except for the surface $Y$, the Picard group is generated
by the lines and conics lying on the surface. In all cases, we decompose
the Picard group into known lattices.
\end{thm*}
\begin{rem*}
The result about a very general member of $\mathcal{X}$ having Picard
rank $16$, with the Picard group generated by the conics, was already
proven by Eklund \cite[Thm 3.5, Cor 7.4]{eklund2010} but in this
paper we prove this using a different method.

The fact that the Fermat quartic has $48$ lines, which generate the
Picard group of rank $20$ is a classical result. We will use that
result in our proof of Theorem \ref{thm:Main Result Final}.
\end{rem*}
We note that Theorem \ref{thm:Main Result Final} fits nicely with
the fact that certain moduli spaces of K3 surfaces whose Picard group
contains a fixed lattice $M$ have dimension $20-\mathrm{rank}(M)$.
I.e., in each of the above, a Picard group of rank $r$, fits nicely
with a $20-r$ dimensional family.

In Section \ref{sec:Lattices} we review the notations and results
we need for lattices. In Section \ref{sec:The-Families-and-Lines}
we start by introducing the notations we will use and review the known
results. We finish the section by introducing the families containing
lines, that we will study in the rest of the paper.

In Section \ref{sec:The-Picard-Group} we look at the Picard group
of the families. First by calculating the Picard number, then by proving
that in each case the Picard group is generated by the conics and
lines. We finish by putting everything together and using the results
from Section \ref{sec:Lattices} to prove Theorem \ref{thm:Main Result Final}.
\begin{note*}
Throughout this paper, most calculations, point counting and lattice
manipulation were done using the computer algebra package Magma \cite[V2.21-4 ]{Magma}.
\end{note*}

\section{Lattices\label{sec:Lattices}}

In this paper a \emph{lattice}, $L$,\emph{ }is a free abelian group
of finite rank equipped with a symmetric, non-degenerate, bilinear
form $\left\langle \,,\,\right\rangle :L\times L\to\zz$. We say it
has \emph{signature} $\left(b_{+},b_{-}\right)$ if for some basis
$\{e_{i}\}$ of $L\otimes_{\mathbb{Z}}\mathbb{R}$ we have 
\[
\left\langle e_{i},e_{j}\right\rangle =\begin{cases}
1 & i=j,i\in\{1,\dots,b_{+}\}\\
-1 & i=j,i\in\{b_{+}+1,\dots,b_{+}+b_{-}\}\\
0 & i\neq j
\end{cases}.
\]
A lattice is \emph{positive definite} if it has signature $(b_{+},0)$,
\emph{negative definite} if it has signature $(0,b_{-})$, and \emph{indefinite}
otherwise. A lattice, $L$, is \emph{even} if $\left\langle x,x\right\rangle \in2\zz$
for all $x\in L$. Let $\{e_{i}\}$ be a basis for $L$, then a \emph{Gram
matrix} of $L$ (with respect to $\left\{ e_{i}\right\} $) is the
matrix $\left(\left\langle e_{i},e_{j}\right\rangle \right)_{i,j}$.
The \emph{discriminant} of $L$, denoted $\mathrm{Disc}(L)$, is the
determinant of a Gram matrix, which is invariant under change of basis.
A lattice is \emph{unimodular }if it has discriminant $\pm1$. 
\begin{example*}
Consider the following Dynkin diagrams:
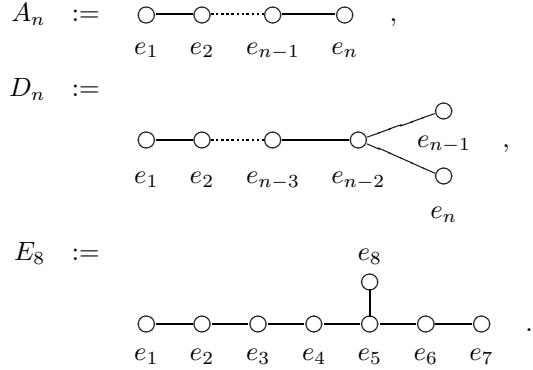
\begin{eqnarray*}
A_{n} & := & \xyR{0.1pc}\xyC{0.5pc}\xymatrix{*+[o][F]{\,}\ar@{-}[r] & *+[o][F]{\,}\ar@{.}[r] & *+[o][F]{\,}\ar@{-}[r] & *+[o][F]{\,} & ,\\
e_{1} & e_{2} & e_{n-1} & e_{n}
}
\\
D_{n} & := & \xyR{0.1pc}\xyC{0.5pc}\xymatrix{\\
 &  &  &  & *+[o][F]{\,}\\
*+[o][F]{\,}\ar@{-}[r] & *+[o][F]{\,}\ar@{.}[r] & *+[o][F]{\,}\ar@{-}[r] & *+[o][F]{\,}\ar@{-}[ru]\ar@{-}[rd] & e_{n-1} & ,\\
e_{1} & e_{2} & e_{n-3} & e_{n-2} & *+[o][F]{\,}\\
 &  &  &  & e_{n}
}
\\
E_{8} & := & \xyR{0.1pc}\xyC{0.5pc}\xymatrix{ &  &  &  & e_{8}\\
 &  &  &  & *+[o][F]{\,}\\
\\
*+[o][F]{\,}\ar@{-}[r] & *+[o][F]{\,}\ar@{-}[r] & *+[o][F]{\,}\ar@{-}[r] & *+[o][F]{\,}\ar@{-}[r] & *+[o][F]{\,}\ar@{-}[r]\ar@{-}[uu] & *+[o][F]{\,}\ar@{-}[r] & *+[o][F]{\,} & .\\
e_{1} & e_{2} & e_{3} & e_{4} & e_{5} & e_{6} & e_{7}
}
\end{eqnarray*}
Each diagram defines a (root) lattice, with basis $\left\{ e_{i}\right\} $
and bilinear form 
\[
\left\langle e_{i},e_{j}\right\rangle =\begin{cases}
2 & i=j\\
-1 & \xyR{0.1pc}\xyC{0.5pc}\xymatrix{*+[o][F]{\,}\ar@{-}[r] & *+[o][F]{\,}\\
e_{i} & e_{j}
}
\\
0 & \mathrm{otherwise}
\end{cases}.
\]

Another example of a lattice is the \emph{hyperbolic plane lattice},
denoted $U$, which is the unique (up to isomorphism) rank $2$ even
indefinite unimodular lattice. For some basis, its Gram matrix is
\[
\begin{pmatrix}0 & 1\\
1 & 0
\end{pmatrix}.
\]
\end{example*}
Given a lattice $L$ with basis $\{e_{i}\}$ and $m\in\zz$, we denote
by $L\left\langle m\right\rangle $ the lattice with basis $\{e_{i}\}$
and bilinear form $\left\langle e_{i},e_{j}\right\rangle _{L\left\langle m\right\rangle }=m\left\langle e_{i},e_{j}\right\rangle _{L}$.
By abuse of notation, we denote the rank one lattice with bilinear
form $\left\langle e_{1},e_{1}\right\rangle =m$ by $\left\langle m\right\rangle $.
If $L_{1}$ and $L_{2}$ are two lattices with basis $\{e_{i}\}$,
$\{f_{i}\}$ respectively, we denote by $L_{1}\oplus L_{2}$ the lattice
with basis $\{e_{i}\}\sqcup\{f_{i}\}$ and bilinear form given by
$\left\langle e_{i},f_{j}\right\rangle =0$. We will say that a lattice
$L$ decomposes into $L_{1},\dots,L_{n}$ if $L\cong L_{1}\oplus\cdots\oplus L_{n}$.

We say a lattice $L_{1}$ is a \emph{sublattice }of a lattice $L_{2}$
if it is a subset of $L_{2}$ and if the bilinear form of $L_{2}$
restricted on $L_{1}$ agrees with the bilinear form of $L_{1}$.
A sublattice is said to be \emph{primitive }if $L_{2}/L_{1}$ is torsion
free. If $L_{1}$ is a full-rank sublattice of $L_{2}$, i.e. $\mathrm{rank}(L_{1})=\mathrm{rank}(L_{2})$,
then we call $L_{2}$ an \emph{overlattice} of $L_{1}$. Note that
in such case $\mathrm{Disc}(L_{1})/\mathrm{Disc}(L_{2})=[L_{2}:L_{1}]^{2}$. 

In Section 4 we try to find a decomposition of lattices into $A_{n}\left\langle m\right\rangle ,D_{n}\left\langle m\right\rangle ,E_{8}\left\langle m\right\rangle $
and $U\left\langle m\right\rangle $, to do so we will need some extra
invariants. We may extend the bilinear form $\left\langle \,,\,\right\rangle $
on $L$ $\qq$-linearly to $L\otimes\qq$ and define the \emph{dual
lattice }(which is often not a lattice with respect to our definition):
\[
L^{*}:=\Hom(L,\zz)\cong\left\{ x\in L\otimes\qq:\left\langle x,y\right\rangle \in\zz\,\forall y\in L\right\} .
\]

\begin{defn}
The \emph{discriminant group} of a lattice $L$ is the finite abelian
group $A_{L}:=L^{*}/L$. We denote by $\ell(A_{L})$ the minimal number
of generators of $A_{L}$. 

The discriminant group comes with a bilinear form, $b_{L}:A_{L}\times A_{L}\to\qq/\zz$
defined by $b_{L}(x+L,y+L)\mapsto\left\langle x,y\right\rangle _{L^{*}}\mod\zz$

For even lattices, we define the \emph{discriminant form}, $q_{L}:A_{L}\to\mathbb{Q}/2\mathbb{Z}$
by $x+L\mapsto\left\langle x,x\right\rangle _{L^{*}}\mod2\mathbb{Z}$.
\end{defn}
The following theorem of Nikulin will help identify the lattices we
will find:
\begin{thm}[{Nikulin \cite[Cor. 1.13.3]{nikulin1980integral}}]
\label{thm:Nikulin-Iso}If a lattice $L$ is even, indefinite with
$\mathrm{rank}(L)>\ell(A_{L})+2$, then $L$ is determined up to isometry
by its rank, signature and discriminant form.
\end{thm}
With that theorem in mind, we write down in Table \ref{tab:Lattices info}
a summary of the rank, signature and discriminant form for the lattices
$U,\,E_{8},\,A_{n}\left\langle m\right\rangle ,\,D_{n}\left\langle m\right\rangle $
and $\left\langle 2m\right\rangle $. 

\begin{center}
{\small{}}
\begin{table}[H]
\begin{centering}
{\small{}}%
\begin{tabular}{c|c|c|c|c|c|c|}
 & {\small{}$U$} & {\small{}$E_{8}$} & {\small{}$A_{n}\left\langle m\right\rangle $} & {\small{}$D_{2n}\left\langle m\right\rangle $} & {\small{}$D_{2n+1}\left\langle m\right\rangle $} & {\small{}$\left\langle 2m\right\rangle $}\tabularnewline
\hline 
{\small{}Rank} & {\small{}$2$} & {\small{}$8$} & {\small{}$n$} & {\small{}$2n$} & {\small{}$2n+1$} & {\small{}$1$}\tabularnewline
\hline 
{\small{}Sgn} & {\small{}$(1,1)$} & {\small{}$(8,0)$} & $\begin{cases}
(n,0) & m>0\\
(0,n) & m<0
\end{cases}$ & {\small{}$\begin{cases}
(2n,0) & m>0\\
(0,2n) & m<0
\end{cases}$} & {\small{}$\begin{cases}
(2n+1,0) & m>0\\
(0,2n+1) & m<0
\end{cases}$} & {\small{}$\begin{cases}
(1,0) & m>0\\
(0,1) & m<0
\end{cases}$}\tabularnewline
\hline 
{\small{}Disc} & {\small{}$1$} & {\small{}$1$} & {\small{}$(n+1)\cdot m^{n}$} & {\small{}$4\cdot m^{2n}$} & {\small{}$4\cdot m^{2n+1}$} & {\small{}$2m$}\tabularnewline
\hline 
$A_{L}$ & $-$ & $-$ & $C_{(n+1)m}\times C_{m}^{n-1}$ & $C_{2m}^{2}\times C_{m}^{2n-2}$ & $C_{4m}\times C_{m}^{2n}$ & $C_{2m}$\tabularnewline
\hline 
$q_{L}$ & $-$ & $-$ & $\left\{ \frac{n}{(n+1)m},\frac{2}{m},\frac{n(n-1)}{m}\right\} $ & $\left\{ \frac{2}{2m},\frac{n}{2m},\frac{2}{m}\right\} $ & $\left\{ \frac{2n+1}{4m},\frac{2}{m},\frac{2n}{m}\right\} $ & $\left\{ \frac{1}{2m}\right\} $\tabularnewline
\hline 
\end{tabular}{\small{}\caption{Invariants of Lattices\label{tab:Lattices info}}
}
\par\end{centering}{\small \par}
\end{table}
\par\end{center}{\small \par}

\noindent The row $q_{L}$ lists the values of $q_{L}(x_{i})$ where
$x_{i}$ are chosen generators of $A_{L}$, i.e., $A_{L}=\left\langle x_{1}\right\rangle \times\dots\times\left\langle x_{\ell(A)}\right\rangle $.
Therefore it only encodes partial information of the discriminant
form and not the whole of it, but it encodes enough to rule out (in
most cases) whether a summand occurs. As $U$ and $E_{8}$ have trivial
discriminant group, we use following theorem of Nikulin to identify
copies of $U$ and $E_{8}$ sitting inside a given lattice.
\begin{thm}[{Nikulin \cite[Cor 1.13.15]{nikulin1980integral}}]
\label{thm:Nikulin-E8-U}Let $L$ be an even lattice of signature
$(b_{+},b_{-})$.

\begin{itemize}
\item If $b_{+}\geq1,b_{-}\geq1$ and $b_{+}+b_{-}\geq3+\ell(A_{L})$ then
$L\cong U\oplus T$ for some $T$.
\item If $b_{+}\geq1,b_{-}\geq8$ and $b_{+}+b_{-}\geq9+\ell(A_{L})$ then
$L\cong E_{8}\left\langle -1\right\rangle \oplus T$ for some $T$.
\end{itemize}
\end{thm}
We note that we can not always have a decomposition of lattices into
$A_{n}\left\langle m\right\rangle ,D_{n}\left\langle m\right\rangle ,E_{8}\left\langle m\right\rangle $
and $U\left\langle m\right\rangle $. When this happens, we express
our lattices as full rank sublattices of a lattice composed of $A_{n}\left\langle m\right\rangle ,D_{n}\left\langle m\right\rangle ,E_{8}\left\langle m\right\rangle $
and $U\left\langle m\right\rangle $. For this we will use:
\begin{thm}
\emph{\cite[Prop 1.4.1]{nikulin1980integral}}\label{thm:Niikulin-overlattice}
Let $L$ be an even lattice. Then there is a natural bijection between
isotropic subgroups $G$ of $A_{L}$ (subgroups on which the discriminant
form $q_{L}$ satisfies $q_{L}(g)=0$ for all $g\in G$) and overlattices
$L_{G}$ of $L$. 

Furthermore, the discriminant form $q_{L_{G}}$ is given by the discriminant
form $q_{L}$ restricted to $G^{\perp}/G$, where orthogonality is
with respect to $b_{L}$. 
\end{thm}

\section{The Families and Lines\label{sec:The-Families-and-Lines}}

We will be studying the variety $\mathcal{X}\subset\pp_{[x,y,z,w]}^{3}\times\pp_{[A,B,C,D,E]}^{4}$
defined by the following equation over $\overline{\qq}$
\[
A(x^{4}+y^{4}+z^{4}+w^{4})+Bxyzw+C(x^{2}y^{2}+z^{2}w^{2})+D(x^{2}z^{2}+y^{2}w^{2})+E(x^{2}w^{2}+y^{2}z^{2})=0.
\]
We view $\mathcal{X}$ as a family of quartic surfaces over $\pp^{3}$
parametrised by points $[A,B,C,D,E]$ in $\pp^{4}$.
\begin{notation*}
We will use $X_{p}$ and $[A,B,C,D,E]$ to denote the quartic surface
parametrised by the point $p=[A,B,C,D,E]\in\pp^{4}$. 
\end{notation*}
\begin{note*}
If $X_{p}$ is a smooth quartic surface, then it is a K3 surface.
\end{note*}
Consider the group $\Omega$ acting on $\pp^{3}\times\pp^{4}$ generated
by the following five elements: the point $[x,y,z,w,A,B,C,D,E]$ is
sent to
\begin{itemize}
\item $[x,y,z,-w,A,-B,C,D,E]$,
\item $[x,y,w,z,A,B,C,E,D]$,
\item $[x,z,y,w,A,B,D,C,E]$,
\item $[x,y,iz,iw,A,-B,C,-D,-E]$,
\item $[x-y,x+y,z-w,z+w,2A+C,8(D-E),12A-2C,B+2D+2E,-B+2D+2E].$
\end{itemize}
\noindent Denote these five elements by $\phi_{1},\phi_{2},\phi_{3},\phi_{4}$
and $\phi_{5}$ respectively. The group $\Omega$ fixes $\mathcal{X}$.
While it is a rather large group with order $2^{4}\cdot6!$, we can
pick out a normal subgroup $\Gamma$, which is generated by the following
four elements
\begin{itemize}
\item $\gamma_{1}:=\phi_{3}\phi_{4}^{2}\phi_{3}\phi_{5}^{2}$,
\item $\gamma_{2}:=\phi_{4}^{2}\phi_{3}\phi_{5}^{2}\phi_{3}$,
\item $\gamma_{3}:=\phi_{4}^{2}$\label{pg:gamma_3},
\item $\gamma_{4}:=\phi_{3}\phi_{4}^{2}\phi_{3}$.
\end{itemize}
The group $\Gamma$ consists of all elements of $\Omega$ which fix
$\pp_{[A,B,C,D,E]}^{4}$ in $\pp^{3}\times\pp^{4}$. In particular
upon picking a point $p\in\pp^{4}$ we have that $\Gamma$ is a subgroup
of $\mathrm{Aut}(X_{p})$ (when projecting the elements of $\Gamma$
onto the $\pp_{[x,y,z,w]}^{3}$ component). Explicitly, when regarding
$\Gamma$ as acting on $\pp^{3}$, we have that its generators are
\[
[x,y,z,w]\mapsto\begin{cases}
[y,x,w,z] & \gamma_{1}\\{}
[z,w,x,y] & \gamma_{2}\\{}
[x,y,-z,-w] & \gamma_{3}\\{}
[x,-y,z,-w] & \gamma_{4}
\end{cases}.
\]
From this we know that $\Gamma\cong C_{2}^{4}$. We calculate that
$\Omega/\Gamma\cong S_{6}$, but $\Omega\ncong C_{2}^{4}\times S_{6}$
because in particular $\Omega$ has trivial centre. We now consider
the cases when $X_{p}$ is not a smooth surface using the following
proposition taken from \cite[Prop 2.1]{eklund2010}.
\begin{prop}
\label{prop:Singular Conditions}Let $p=[A,B,C,D,E]\in\pp^{4}$. The
surface $X_{p}$ is singular if and only if
\[
\Delta\cdot A\cdot q_{+C}\cdot q_{-C}\cdot q_{+D}\cdot q_{-D}\cdot q_{+E}\cdot q_{-E}\cdot p_{+0}\cdot p_{+1}\cdot p_{+2}\cdot p_{+3}\cdot p_{-0}\cdot p_{-1}\cdot p_{-2}\cdot p_{-3}=0,
\]
where:
\begin{equation}
\Delta=16A^{3}+AB^{2}-4A(C^{2}+D^{2}+E^{2})+4CDE\label{eq:Delta}
\end{equation}
\begin{flalign*}
q_{+C}= & 2A+C & q_{+D}= & 2A+D & q_{+E}= & 2A+E\\
q_{-C}= & 2A-C & q_{-D}= & 2A-D & q_{-E}= & 2A-E\\
p_{+0}= & 4A+B+2C+2D+2E &  &  & p_{-0}= & 4A-B+2C+2D+2E\\
p_{+1}= & 4A+B+2C-2D-2E &  &  & p_{-1}= & 4A-B+2C-2D-2E\\
p_{+2}= & 4A+B-2C+2D-2E &  &  & p_{-2}= & 4A-B-2C+2D-2E\\
p_{+3}= & 4A+B-2C-2D+2E &  &  & p_{-3}= & 4A-B-2C-2D+2E.
\end{flalign*}
\end{prop}
\begin{defn}
The surface $S_{3}=\left\{ \Delta=0\right\} \subset\pp^{4}$ is the
\textit{Segre cubic}\textit{\emph{. We shall refer to the $15$ hyperplanes
in $\pp^{4}$ defined by the $15$ equations 
\[
\left\{ A,p_{\pm j},q_{\pm\alpha}:\alpha\in\left\{ C,D,E\right\} ,j\in\left\{ 0,1,2,3\right\} \right\} 
\]
 above as the }}\textit{singular hyperplanes}\textit{\emph{. }}
\end{defn}
\noindent The Segre cubic has $10$ singular points, namely:
\[
[1,0,-2,-2,2],[1,0,-2,2,-2],[1,0,2,-2,-2],[1,0,2,2,2],
\]
\[
[0,-2,1,0,0],[0,2,1,0,0],[0,-2,0,1,0],[0,2,0,1,0],[0,-2,0,0,1],\mathrm{\,and\,}[0,2,0,0,1].
\]
We shall denote these $10$ points by $q_{i}$, $i\in[1,\dots,10]$,
as ordered above. These $10$ points have associated quartics in $\pp^{3}$,
which turns out to be quadrics of multiplicity two. We denote the
quadric associated to the point $q_{i}$ by $Q_{i}$. Explicitly they
are:
\[
x^{2}-y^{2}-z^{2}+w^{2}=0,\,x^{2}-y^{2}+z^{2}-w^{2}=0,\,x^{2}+y^{2}-z^{2}-w^{2}=0,\,x^{2}+y^{2}+z^{2}+w^{2}=0,
\]
\[
xy-zw=0,\,xy+zw=0,\,xz-yw=0,\,xz+yw=0,\,xw-yz=0,\mathrm{\,and\,}xw+yz=0.
\]

\selectlanguage{english}%
\begin{defn}
We define the \emph{Nieto quintic}, \foreignlanguage{british}{\emph{$N_{5}\subseteq\pp_{[A,B,C,D,E]}^{4}$,}
}by the equation
\[
4A^{3}(48A^{2}-B^{2})-A(32A^{2}-B^{2})(C^{2}+D^{2}+E^{2})-4A(C+D+E)(C+D-E)(C-D+E)(-C+D+E)+B^{2}CDE=0.
\]
\end{defn}
\selectlanguage{british}%
The Nieto quintic was studied by Barth and Nieto when they were looking
at K3 surfaces in $\mathcal{X}$ containing lines. In particular,
they proved in \cite[Section 7 and 8]{BarthNieto} the following proposition.
\selectlanguage{english}%
\begin{prop}
\label{prop:Barth=000026Nieto} Let $p\in\pp^{4}$, then the surface
$X_{p}$ contains a line, $L$, if and only if $p$ is in $N_{5}$
or in one of the $10$ tangent cones to the isolated singular points
of $N_{5}$ (i.e., the $10$ nodes of $S_{3}$).

In the case where $p$ lies on the tangent cone of $q_{i}$, then
$L$ lies on $Q_{i}$. 
\end{prop}
\selectlanguage{british}%
As Eklund studies in detail the K3 surfaces defined by a point lying
on the Nieto quintic \cite{eklund2010}, we study here those surfaces
defined by a point lying on the $10$ tangent cones. We first make
the following remark:
\selectlanguage{english}%
\begin{rem*}
The four roots of the equation $f=x^{4}+cx^{2}+1$ are of the form
\[
\alpha=\frac{1}{2}\left(\sqrt{-c+2}+\sqrt{-c-2}\right).
\]
To see this, note that $\alpha^{2}=\frac{1}{2}(-c+\sqrt{c^{2}-4})$
which solves $y^{2}+cy+1$.
\end{rem*}
\begin{prop}
\label{prop:Conatins 8 Lines}\foreignlanguage{british}{Let $p\in\pp^{4}$
lie on one of the $10$ tangent cones to the isolated singular points
of $N_{5}$, away from $N_{5}$ and the $15$ singular planes. Then
the surface $X_{p}$ contains eight lines. In the case where $p$
lies on a unique tangent cone, $X_{p}$ contains exactly eight lines.}
\end{prop}
\selectlanguage{british}%
\begin{proof}
If $p\in\pp^{4}\setminus N_{5}$ lies on a unique tangent cone, say
$q_{i}$, then by Proposition \ref{prop:Barth=000026Nieto} all the
lines lying on $X_{p}$ must be lines lying on $Q_{i}$.

We first prove that when $p=[A,B,C,D,E]\in\pp^{4}$ lies on the tangent
cone of the point $q_{6}$, there are exactly eight lines lying on
$Q_{6}\cap X_{p}\subseteq\pp^{3}$. By \cite[3.2]{BarthNieto}, we
have that $p$ satisfies the equation $AB-2AC+DE=0$. The quadric
$Q_{6}:xy+wz=0$ has the following lines (for any $\alpha\in K^{*}$)

\begin{itemize}
\item $x+\alpha z=y-\alpha^{-1}w=0$,
\item $x+\alpha w=y-\alpha^{-1}z=0$,
\item $x=z=0$,
\item $x=w=0$,
\item $y=z=0$,
\item $y=w=0$.
\end{itemize}
Note that the last four lines can not lie on $X_{p}$, as $p$ does
not lie on the $15$ singular planes (hence $A\neq0$). Now $X_{p}\cap\{x+\alpha z=y-\alpha^{-1}w=0\}$
is defined by the equations: 
\begin{eqnarray*}
x+\alpha z & = & 0,\\
y-\alpha^{-1}w & = & 0,\\
(A\alpha^{4}+D\alpha^{2}+A)\left(z^{4}+\frac{w^{4}}{\alpha^{4}}\right)+\left(E\alpha^{4}+(2C-B)\alpha^{2}+E\right)\frac{z^{2}w^{2}}{\alpha^{2}} & = & 0.
\end{eqnarray*}
As $AB-2AC+DE=0$ implies 
\begin{eqnarray*}
E\alpha^{4}+(2C-B)\alpha^{2}+E & = & E\alpha^{4}+\frac{DE}{A}\alpha^{2}+E\\
 & = & \frac{E}{A}(A\alpha^{4}+D\alpha^{2}+A),
\end{eqnarray*}
we have that the last equation becomes 
\begin{eqnarray*}
 &  & (A\alpha^{4}+D\alpha^{2}+A)\left(z^{4}+\frac{w^{4}}{\alpha^{4}}\right)+(A\alpha^{4}+D\alpha^{2}+A)\frac{Ez^{2}w^{2}}{A\alpha^{2}}\\
 & = & (A\alpha^{4}+D\alpha^{2}+A)(z^{4}+\frac{Ez^{2}w^{2}}{A\alpha^{2}}+\frac{w^{4}}{\alpha^{4}}).
\end{eqnarray*}
This is identically zero if and only $A\alpha^{4}+D\alpha^{2}+A=0$.
Hence there are exactly four lines of the form $x+\alpha z=y-\alpha^{-1}w=0$
on $X_{p}$, corresponding to the four roots of $A\alpha^{4}+D\alpha^{2}+A=0$.
We can run through exactly the same process for lines of the form
$x+\alpha w=y-\alpha^{-1}z=0$ and find that this time $\alpha$ needs
to solve $A\alpha^{4}+E\alpha^{2}+A=0$. Hence by letting
\[
\alpha=\frac{\sqrt{A}}{2A}\left(\sqrt{q_{-D}}+\sqrt{-q_{+D}}\right)
\]
\[
\beta=\frac{\sqrt{A}}{2A}\left(\sqrt{q_{-E}}+\sqrt{-q_{+E}}\right)
\]
we have the eight lines\begin{multicols}{2}

\begin{itemize}
\item $x+\alpha z=y-\alpha^{-1}w=0$, 
\item $x-\alpha z=y+\alpha^{-1}w=0$,
\item $x+\beta w=y-\beta^{-1}z=0$, 
\item $x-\beta w=y+\beta^{-1}z=0$, 
\item $x+\alpha^{-1}z=y-\alpha w=0$,
\item $x-\alpha^{-1}z=y+\alpha w=0$,
\item $x+\beta^{-1}w=y-\beta z=0$,
\item $x-\beta^{-1}w=y+\beta z=0$,
\end{itemize}
\end{multicols} which lie on our surface $X_{p}$, and these are
the only lines on $X_{p}\cap Q_{6}$. \foreignlanguage{english}{To
finish the proof, }we use the group $\Omega$ acting on $\pp^{3}\times\pp^{4}$.
This group permutes the $15$ singular planes and the $10$ points
$q_{1}$ as follows:

\begin{itemize}
\item $\phi_{1}$ acts as $(p_{+0},p_{-0})(p_{+1},p_{-1})(p_{+2},p_{-2})(p_{+3},p_{-3})$
and as $(q_{5},q_{6})(q_{7},q_{8})(q_{9},q_{10})$,
\item $\phi_{2}$ acts as $(q_{+D},q_{+E})(q_{-D},q_{-E})(p_{+2},p_{+3})(p_{-2},p_{-3})$
and as $(q_{1},q_{2})(q_{7},q_{9})(q_{8},q_{10})$,
\item $\phi_{3}$ acts as $(q_{+C},q_{+D})(q_{-C},q_{-D})(p_{+1},p_{+2})(p_{-1},p_{-2})$
and as $(q_{2},q_{3})(q_{5},q_{7})(q_{6},q_{8})$,
\item $\phi_{4}$ acts as $(q_{+D},q_{-D})(q_{+E},q_{-E})(p_{+0},p_{-1})(p_{-0},p_{+1})(p_{+2},p_{-3})(p_{-2},p_{+3})$
and as $(q_{1},q_{2})(q_{3},q_{4})(q_{5},q_{6})$,
\item $\phi_{5}$ acts as $(A,q_{+C})(q_{+D},p_{+0})(q_{-D},p_{-1})(q_{+E},p_{-0})(q_{-E},p_{+1})(p_{+2},p_{-3})$
and as $(q_{1},q_{5})(q_{2},q_{6})(q_{7},q_{10})$.
\end{itemize}
Hence by applying the appropriate element $\phi\in\Omega$ on the
above eight lines, we get the equations of the eight lines lying on
the surface $X_{\phi(p)}$. We have listed the equations of the lines
in Table \ref{tab:Equations-of-lines} of Appendix \ref{sec:Eqn of lines}.
\end{proof}
\selectlanguage{english}%
\label{pg:interchanging_lines}Using the fact that the eight lines
comes from the two different rulings of the quadric (one set using
$\alpha$, the other $\beta$), it is not hard to see that the lines
come in two sets of four skew lines. Furthermore each line from one
set intersects each of the four lines in the other set.

\selectlanguage{british}%
Finally, using the explicit equations, we note that given two (not
necessarily distinct) lines in one set, $L_{1}\,\mathrm{and\,}L_{2}$,
and two in the other set $M_{1}\mathrm{\,and\,}M_{2}$, there exists
a unique $\gamma\in\Gamma$ interchanging $L_{1}$ with $L_{2}$ and
$M_{1}$ with $M_{2}$.

We can use Proposition \ref{prop:Conatins 8 Lines} to find various
families containing $8,16,24,32$ and $48$ lines.\\

\begin{lem}
~\label{lem:Number of lines}

\begin{itemize}
\item A very general surface in the family $[A,(DE-2AC)/A,C,D,E]$ contains
exactly $8$ lines. We denote this family by $\mathcal{X}_{C,D,E}$,
\item A very general surface in the family $[A,0,C,D,2AC/D]$ contains exactly
$16$ lines. We denote this family by $\mathcal{X}_{C,D}$,
\item A very general surface in the family $[A,B(2A-B)/A,B,B,B]$ contains
exactly $24$ lines. We denote this family by $\mathcal{X}_{B}$,
\item A very  general surface in the family $[A,0,C,0,0]$ contains exactly
$32$ lines. We denote this family by $\mathcal{X}_{C}$,
\item The surface $[\sqrt{-3},12(\sqrt{-3}-1),6,6,-6]$ contains exactly
$32$ lines. We denote this surface by $Y$,
\item The Fermat quartic $[1,0,0,0,0]$ contains exactly $48$ lines. We
denote this surface by $F_{4}$. 
\end{itemize}
Up to an action of $\Omega$, there are no other families whose very
general member is smooth and lies on the tangent cones to one of the
points $q_{i}$.

\end{lem}
\begin{proof}
Note that for each family, a very general point will not be on $N_{5}$,
hence if for each family a very general member lie on $m$ distinct
tangent cones, then it will contain $8m$ lines as claimed. 

Recall that $\Omega$ acts on the 10 points $q_{i}$, and hence on
the $10$ tangent cones. For each $m\in\{1,\dots,10\}$, we find representatives
of the action of $\Omega$ on sets of size $m$. For example, when
$m=2$, as $\Omega$ is two-transitive, we have the representative
$\{q_{1},q_{2}\}$, for $m=3$, we have two representative $\{q_{1},q_{2},q_{3}\}$
and $\{q_{2},q_{4},q_{5}\}$. Starting from $m=10$ to $1$, for each
representative we intersect the corresponding tangent cones. We look
at its irreducible components and discard any that is a subset of
$\mathcal{L}$ (the union of the $15$ singular hyperplanes), any
component remaining give us a family that we list. This also verifies
that our list is complete. This calculation is available online \cite{ThesisCode}.
\end{proof}
We illustrate how the families fit together with the following diagram.
The lines show which family is a subfamily of another family.

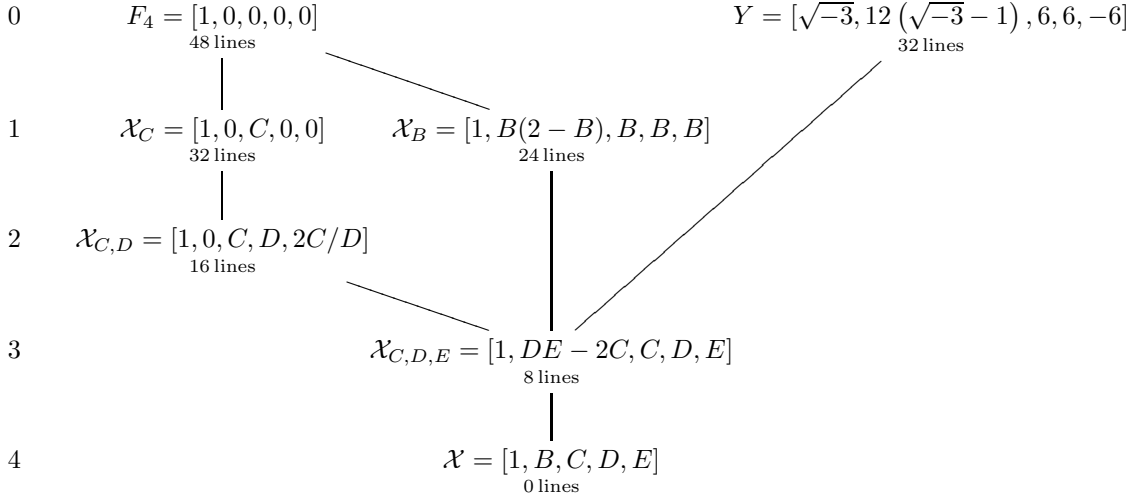
\selectlanguage{english}%
\[
\xyR{1.5pc}\xyC{-0.5pc}\xymatrix{\mathrm{Dimension}\\
0 & \underset{48\mathrm{\,lines}}{F_{4}=[1,0,0,0,0]} &  & \underset{32\,\mathrm{lines}}{Y=[\sqrt{-3},12\left(\sqrt{-3}-1\right),6,6,-6]}\ar@{-}[dddl]\\
1 & \underset{32\,\mathrm{lines}}{\mathcal{X}_{C}=[1,0,C,0,0]}\ar@{-}[u] & \underset{24\,\mathrm{lines}}{\mathcal{X}_{B}=[1,B(2-B),B,B,B]}\ar@{-}[ul]\\
2 & \underset{16\,\mathrm{lines}}{\mathcal{X}_{C,D}=[1,0,C,D,2C/D]}\ar@{-}[u]\\
3 &  & \underset{8\,\mathrm{lines}}{\mathcal{X}_{C,D,E}=[1,DE-2C,C,D,E]}\ar@{-}[uu]\ar@{-}[ul]\\
4 &  & \underset{0\,\mathrm{lines}}{\mathcal{X}=[1,B,C,D,E]}\ar@{-}[u]
}
\]

\selectlanguage{british}%

\section{The Picard Group\label{sec:The-Picard-Group}}

We now turn to proving that the Picard rank of the families given
above are those claimed by (the summarised) Theorem \ref{thm:Main Result Final}.
Note that we already know this to be true for the Fermat quartic,
$x^{4}+y^{4}+z^{4}+w^{4}$ (see for example \cite{NSofFermat}) and
the family, $\mathcal{X}$, parameterised by $\pp^{4}$ (\cite{eklund2010}).
To achieve this, for each family we will bound the rank from below
and above. To bound the Picard rank from below, we use the following
theorem proven in \foreignlanguage{english}{\cite[Thm 4.3]{eklund2010}.}
\selectlanguage{english}%
\begin{thm}
\label{thm:320-Conics}A very general K3 surface $X$ from the family
$\mathcal{X}$ contains at least $320$ smooth conics. 
\end{thm}
The equations of the conics can be listed explicitly in terms of the
point $[A,B,C,D,E]\in\pp^{4}$ giving the surface $X$, more details
can be found in \cite{fbouyer_mono}. As the lines and conics are
elements of $\mathrm{Pic}(X)$, they form a sublattice of it. Hence
by using the explicit equations of the $320$ conics and $8m$ lines,
we can calculate their intersection matrix. The rank of said matrix,
which is the rank of the sublattice generated by lines and conics,
is a lower bound to the rank of the Picard group. 

To calculate an upper bound, we use three main ideas:
\selectlanguage{british}%

\subsection{Reduction at a good prime \label{subsec:Reduction}}

The first idea is due to Van Luijk \cite{vLuikPaper} which we briefly
recap here.
\begin{thm}
\label{thm:Reduction}Let $X$ be a \emph{K3 }surface defined over
a number field $K$. Choose a finite prime $\mathfrak{p}\subseteq\mathcal{O}_{K}$
of good reduction for $X$. Let $R=\left(\mathcal{O}_{K}\right)_{\mathfrak{p}}$
and $k$ its residue field. Fix an algebraic closure $\overline{K}$
of $K$, $\overline{R}$ the integral closure of $R$ in $\overline{K}$,
and let $\overline{k}=\overline{R}/\mathfrak{p}$ be the algebraic
closure of $k$. There are natural injective homomorphisms 
\[
\mathrm{NS}(X_{\overline{K}})\otimes\qq_{\ell}\hookrightarrow\mathrm{NS}(X_{\overline{k}})\otimes\qq_{\ell}\hookrightarrow H_{\mathrm{\acute{e}t}}^{2}(X_{\overline{k}},\qq_{\ell}(1))
\]
of finite dimensional vector space over $\qq_{\ell}$. The second
injection respects the Galois action $\mathrm{Gal}(\overline{k}/k)$.
\end{thm}
\begin{prop}
Let $X$ be a \emph{K3 }surface defined over a finite field $\ff_{q}$
with $q=p^{r}$. Let $F_{q}:X\to X$ be the \emph{absolute Frobenius
map of $X$, }which acts on the identity on points, and by $x\mapsto x^{p}$
on the structure sheaf. Set $\Phi_{q}=F_{q}^{r}$ and let $\Phi_{q}^{*}$
denote the automorphism on $H_{\mathrm{\acute{e}t}}^{2}(\overline{X},\qq_{\ell})$
induced by $\Phi_{q}\times1$ acting on $X_{\overline{\ff}_{q}}$.
Then the rank of $\mbox{NS}(X_{\overline{\ff}_{q}})$ is bounded above
by the number of eigenvalues $\lambda$ of $\Phi_{q}^{*}$ for which
$\lambda/q$ is a root of unity (counted with multiplicity).
\end{prop}
Hence, given a K3 surface over a number field $K$, its Picard rank,
$\rho(X_{\overline{K}})$, is bounded above by eigenvalues in a certain
form of $\Phi_{q}^{*}$. Such eigenvalues can be read off from the
characteristic polynomial, $f_{q}(x)$, of $\Phi_{q}^{*}$. To calculate
the characteristic polynomial we use the Lefschetz formula:
\[
\mathrm{Tr}\left(\left(\Phi_{q}^{*}\right)^{i}\right)=\#X_{k}(\ff_{q^{i}})-1-q^{2i},
\]
and the following lemma:
\begin{lem}[Newton's Identity]
Let $V$ be a vector space of dimension $n$ and $T$ a linear operator
on $V$. Let $t_{i}$ denote the trace of $T^{i}$. Then the characteristic
polynomial of $T$ is equal to 
\[
f_{T}(x)=\det(x\cdot\id-T)=x^{n}+c_{1}x^{n-1}+c_{2}x^{n-2}+\dots+c_{n}
\]
 where the $c_{i}$ are given recursively by $c_{1}=-t_{1}$ and 
\[
-kc_{k}=t_{k}+\sum_{i=1}^{k-1}c_{i}t_{k-i}.
\]
\end{lem}
So in theory, since $n=22$ as $X$ is a K3 surface, we can calculate
the characteristic polynomial by counting points over $\ff_{q^{i}}$
for $i=1,\dots,22$. But this is computationally infeasible. To make
the computation more feasible we use the fact that from the Weil conjectures
we have the functional equation
\[
p^{22}f_{q}(x)=\pm x^{22}f_{q}(p^{2}/x).
\]

Second of all, in our cases we have an explicit submodule $M\subseteq\mathrm{NS}X_{\overline{k}}$
of rank $r$, namely the one generated by the lines and conics lying
on $X$. Hence we can calculate the characteristic polynomial $f_{M}(x)$
of Frobenius acting on $M$. Since $f_{M}(x)|f_{q}(x)$, we can compute
two possible polynomials $f_{q,+}(x)$ and $f_{q,-}(x)$ (one for
each possible sign in the functional equation) by counting points
on $X_{k}(\ff_{q^{i}})$ for $i=1,\dots(22-r)/2$. 

Explicitly, suppose $f_{M}(x)=\prod_{j}g_{j}(x)^{e_{j}}$ with $\deg(g_{j})=d_{j}$,
hence $\sum d_{j}e_{j}=r$. Note that $f'_{q}(x)=f_{M}(x)h'(x)+f_{M}'(x)h(x)$,
hence if $e_{j}>1$ then $g_{j}(x)|f'_{q}(x)$, and in general $g_{j}(x)$
divides the $(e_{j}-1)$th derivative of $f_{q}(x)$. Therefore, we
can use the roots of $M$ to construct $r/2$ linear equations in
the $11$ coefficients of $f_{q}(x)$ (by assuming $f_{q}(x)$ satisfies
one of the functional equation). Hence we just need to count points
on $X_{k}(\ff_{q^{i}})$ for $i=1,\dots(22-r)/2$ to be able to use
linear algebra and find the $11$ coefficients of $f_{q}(x)$. Note
that when we assume the negative functional equations, we have in
fact only $10$ coefficients of $f_{q}(x)$, as $c_{11}=0$. Hence,
we end up not using all the information from $f_{M}(x)$, therefore
it is possible to construct $f_{q}(x)$ such that $f_{M}(x)\nmid f_{q}(x)$.
This is a contradiction, meaning that $f_{q}(x)$ satisfies the positive
functional equation and not the negative.

Finally, note that by rescaling $f_{q}(x)$ by $f_{q}(x/p)$, we just
need to count the roots which are also roots of unity.

\subsection{Artin-Tate conjecture \label{subsec:Artin-Tate}}

Unfortunately, as the roots come in conjugate pairs, the above method
can only ever give an even upper bound. The following proposition
can potentially reduce the upper bound by one more than the above
bound.
\begin{prop}
Let $X$ be a \emph{K3} surface defined over a number field $K$ and
let $\mathfrak{p}$ and $\mathfrak{p}'$ be two primes of good reductions.
Suppose that $\rho(\overline{X}_{\mathfrak{p}})=\rho\left(\overline{X}_{\mathfrak{p}'}\right)=n$
but the discriminants $\mathrm{Disc}(\mathrm{NS}(\overline{X}_{\mathfrak{p}}))$
and $\mathrm{Disc}(\mathrm{NS}(\overline{X}_{\mathfrak{p}'}))$ are
different in $\qq^{*}/(\qq^{*})^{2}$. Then $\rho(\overline{X})<n$.
\end{prop}
\begin{proof}
By the above, we know that $\rho(\overline{X})\leq n$. If $\rho(\overline{X})=n$,
then $\mathrm{NS}(\overline{X})$ is a full rank sublattice of $\mathrm{NS}(\overline{X}_{\mathfrak{p}})$
and $\mathrm{NS}(\overline{X}_{\mathfrak{p}'})$. But in that case,
as elements of $\qq^{*}/(\qq^{*})^{2}$, all three discriminants should
be equal, which is a contradiction to the hypothesis.
\end{proof}
As the proposition requires us to calculate the discriminants of $\mathrm{NS}(\overline{X}_{\mathfrak{p}})$
and $\mathrm{NS}(\overline{X}_{\mathfrak{p}'})$ we use the following
conjecture:
\begin{conjecture}[Artin - Tate]
 Let $X$ be a \emph{K3} surface over a finite field $\ff_{q}$.
Let $\rho$ and $\mathrm{Disc}$ denote respectively the rank and
discriminant of the Picard group defined over $\ff_{q}$. Then
\[
\left|\mathrm{Disc}\right|=\frac{\lim_{T\to q}\frac{\Phi(T)}{(T-q)^{\rho}}}{q^{21-\rho}\#\mathrm{Br}(X)}.
\]
Here $\Phi$ is the characteristic polynomial of $\mathrm{Frob}$
on $H_{\mathrm{\acute{e}t}}^{2}(X_{\overline{\ff}_{q}},\qq_{l})$.
Finally, $\mathrm{Br}(X)$ is the Brauer group of $X$. 
\end{conjecture}
In the case when $q$ is odd, then the above conjecture has been proven
to be true (using the fact that it follows from the Tate Conjecture
\cite{MR0414558} which has been proven for K3 surfaces \cite{MR723215,MR819555,MR3103257,MR3265555,pera2013tate}).
Furthermore in the case the conjecture is true we have that $\#\Br(X)$
is a square. Hence, by picking $q$ large enough so that $\rho(X_{q})=\rho(\overline{X}_{q})$,
we can find $\left|\mathrm{Disc}\right|$ as an element of $\qq^{*}/\left(\qq^{*}\right)^{2}$.
\selectlanguage{english}%

\subsection{Only finitely many singular K3 surfaces \label{subsec:rank 19}}

\selectlanguage{british}%
Suppose that a general member of the family $\mathcal{Y}$ has Picard
rank at least $19$ and the family $\mathcal{Y}$ is parameterised
by a one dimensional curve. The third idea uses the fact that, up
to $\overline{\qq}$-isomorphism, there only finitely many K3 surfaces
over $\qq$ which are singular, i.e., with Picard rank $20$. Hence
if a very general member of $\mathcal{Y}$ has Picard rank $20$,
then every member of $\mathcal{Y}$ is singular. Therefore $\mathcal{Y}$
parametrises a set of isomorphic surfaces. If we can show that there
are two $\qq$-surfaces in $\mathcal{Y}$ which are not isomorphic,
then a very general member of the family $\mathcal{Y}$ has Picard
rank at most $19$ (as it can not be $20$).

We implement this by noting that in each of the cases we are interested
in, the Fermat quartic, $F_{4}$, belongs to our family $\mathcal{Y}$.
Furthermore the Fermat quartic is supersingular over algebraically
closed fields of characteristic $3$ mod $4$, i.e., $\rho\left(\overline{F_{4,p}}\right)=22$
for $p\equiv3\mod4$ \cite{MR0225778}. Hence if there is another
surface in $\mathcal{Y}$ with $\rho\left(\overline{X}_{p}\right)=20$
over a prime $p\equiv3\mod4$, then $F_{4}$ and $X$ are not isomorphic
(since their specialisations to the field $\mathrm{\ff}_{p}$ are
not isomorphic, as they have different Picard rank).

With all these tools we tackle the following proposition:
\begin{prop}
~\label{prop:Picard Rank}

\begin{itemize}
\item A very general surface in the family $\mathcal{X}$ has Picard rank
$16$,
\item A very general surface in the family $\mathcal{X}_{C,D,E}$ has Picard
rank $17$,
\item A very general surface in the family $\mathcal{X}_{C,D}$ has Picard
rank $18$,
\item A very general surface in the family $\mathcal{X}_{B}$ has Picard
rank $19$,
\item A very general surface in the family $\mathcal{X}_{C}$ has Picard
rank $19$,
\item The surface $Y$ is singular.
\end{itemize}
\end{prop}
\begin{proof}
To get the lower bound we want to calculate the intersection matrix
of the conics and lines lying on a very general member of each family.
The lines and conics are defined over a degree $2^{10}$ field extension,
hence calculating the intersection matrix is computationally infeasible.
Instead we do the calculations over finite fields. Pick $X$ in one
of the families (call it $\mathcal{X}$) and let $p$ be a prime of
good reduction. Then we know that the conics and lines of $X_{\ff_{p}}$
are defined over $\ff_{p^{2}}$ (due to having explicit equations
and there are only two square classes in $\ff_{p}$) and so we calculate
with ease the intersection matrix. By Theorem \ref{thm:Reduction}
$\mathrm{NS}(X_{\overline{\qq}})\otimes\qq_{\ell}\hookrightarrow\mathrm{NS}(X_{\overline{\ff_{p}}})\otimes\qq_{\ell}$
is injective, so the intersection matrix of the lines and conics over
$\ff_{p^{2}}$ is the same as the intersection matrix of the lines
and conics over $\overline{\qq}$. Furthermore, as the set of surfaces
in $\mathcal{X}_{*}$ which reduce to $X_{\ff_{p}}$ is Zariski open,
the intersection matrix calculated is the same as the intersection
matrix of a very general member of $\mathcal{X}_{*}$.

As the intersection matrix is a large matrix, we have included in
Appendix \ref{sec:Matrices} a full rank minor of the matrix for each
family (in particular, the lower bound is the dimension of said minor).
We work (see \cite{ThesisCode}) through the families in reverse order
from the list above.

\begin{itemize}
\item As $M_{Y}$ has rank $20$, we know that $\rho(Y)=20$ and hence the
surface $Y$ is singular.
\item As $M_{C}$ has rank $19$, we know that a very general surface $X_{C}$
of $\mathcal{X}_{C}$ has $\rho(X_{C})\geq19$. Using the idea in
Subsection~\ref{subsec:rank 19} we see that the surface $X_{0}$,
associated to the point $[1,0,0,0,0]$, is the Fermat quartic so it
is supersingular over $\overline{\ff}_{19}$. On the other hand consider
the surface $X_{5}$, associated to the point $[1,0,5,0,0]$, over
$\ff_{19}$. The characteristic polynomial of Frobenius acting on
conics and lines on $X_{5}$ is $f(x)=(x-1)^{10}(x+1)^{9}$. Hence
we just need to count points over $\ff_{19}$ and $\ff_{19^{2}}$
to find the two possible characteristic polynomials for $\Phi_{19}^{*}$.
We find, after rescaling, $f_{19,+}(x)=\frac{1}{19}(x-1)^{10}(x+1)^{10}(19x^{2}-22x+19)$
and a contradiction for $f_{19,-}(x)=\frac{1}{19}(x-1)^{9}(x+1)^{9}(19x^{4}-22x^{3}-22x+19)$
as $f(x)\nmid f_{19,-}(x)$. As $X_{2}$ is not supersingular, $X_{0}$
and $X_{2}$ are not isomorphic over $\overline{\ff}_{19}$. Therefore
a very general surface in $\mathcal{X}_{C}$ has Picard number $19$.
\item As $M_{B}$ has rank $19$, we know that a very general surface $X_{B}$
of $\mathcal{X}_{B}$ has $\rho(X_{B})\geq19$. Using the idea in
Subsection~\ref{subsec:rank 19} we see that $X_{2}$, associated
to the point $[1,0,0,0,0]$, is the Fermat quartic so is supersingular
over $\overline{\ff}_{19}$. On the other hand consider the surface
$X_{1}$, associated to the point $[1,1,1,1,1]$, over $\ff_{19}$.
The characteristic polynomial of Frobenius acting on conics and lines
on $X_{1}$ is $f(x)=(x-1)^{16}(x+1)^{3}$. After point counting over
$\ff_{19}$ and $\ff_{19^{2}}$ we find the possible two characteristic
polynomials for $\Phi_{19}^{*}$, namely $f_{19,+}(x)=\frac{1}{19}(x-1)^{16}(x+1)^{4}(19x^{2}-18x+19)$
and a contradiction for $f_{19,-}(x)=\frac{1}{19}(x-1)^{15}(x+1)^{3}(19x^{4}-18x^{3}-18x+19)$.
As $X_{1}$ is not supersingular, $X_{2}$ and $X_{1}$ are not isomorphic
over $\overline{\ff}_{19}$. Therefore a very general surface in $\mathcal{X}_{B}$
has Picard number $19$.
\item As $M_{C,D}$ has rank $18$, we know that a very general surface
$X_{C,D}$ of $\mathcal{X}_{C,D}$ has $\rho(X_{C,D})\geq18$. We
use the idea in Subsection~\ref{subsec:Reduction} and find a surface
whose reduction at a prime $p$ gives an upper bound of $18$. To
make point counting easier, we will work over $\ff_{13}$ and the
surface $X_{4,1}$, associated to the point $[1,0,4,1,8]$. Our first
step is to calculate the characteristic polynomial of Frobenius acting
on conics and lines, which is $f_{}(x)=(x-1)^{10}(x+1)^{8}$. After
point counting over $\ff_{13}$ and $\ff_{13^{2}}$ we find the two
possible characteristic polynomials for $\Phi_{13}^{*}$, namely $f_{13,+}(x)=\frac{1}{13}(x-1)^{10}(x+1)^{8}(13t^{4}+12t^{3}+14t^{2}+12t+13)$
and a contradiction for $f_{13,-}(x)=\frac{1}{13}(x-1)^{9}(x+1)^{9}(13t^{4}-14t^{3}+16t^{2}-14t+13)$
(since $f(x)\nmid f_{13,-}(x)$). Hence $\rho(\overline{X}_{4,1})\leq18$,
so a very general surface in $\mathcal{X}_{C,D}$ has Picard number
$18$.
\item As $M_{C,D,E}$ has rank $17$, we know that a very general surface
$X_{C,D,E}$ of $\mathcal{X}_{C,D,E}$ has $\rho(X_{C,D,E})\geq17$.
We use the idea in Subsection~\ref{subsec:Reduction} and find a
surface whose reduction at two primes $p$ and $p'$ gives an upper
bound of $18$. We work with the surface $X_{3,5,7}$, associated
to the point $[1,29,3,5,7]$, over the fields $\ff_{13}$ and $\ff_{19}$.
The characteristic polynomial of Frobenius acting on conics and lines
over $\ff_{13}$ is $f_{13}(x)=(x-1)^{8}(x+1)^{9}$ and over $\ff_{19}$
is $f_{19}(x)=(x-1)^{9}(x+1)^{8}$. We find the following possible
characteristic polynomials (after rescaling):

\begin{center}
\begin{tabular}{|c|>{\centering}p{5.5cm}|>{\centering}p{6cm}|}
\cline{2-3} 
\multicolumn{1}{c|}{} & $f_{+}$ & $f_{-}$\tabularnewline
\hline 
$\ff_{13}$ & $\frac{1}{13}(x-1)^{8}(x+1)^{10}(13x^{4}+22x^{2}+13)$ & $\frac{1}{13}(x-1)^{9}(x+1)^{9}(13x^{4}+26x^{3}+48x^{2}+26x+13)$\tabularnewline
\hline 
$\ff_{19}$ & $\frac{1}{19}(x-1)^{10}(x+1)^{8}(19x^{4}+32x^{3}+42x^{2}+32x+19)$ & $\frac{1}{19}(x-1)^{9}(x+1)^{9}(19x^{4}-6x^{3}+16x^{2}-6x+19)$\tabularnewline
\hline 
\end{tabular}
\par\end{center}

\noindent We then apply the idea in Subsection~\ref{subsec:Artin-Tate},
by working over $\ff_{13^{2}}$ and $\ff_{19^{2}}$. We find that,
up to squares, the discriminants are as follow:

\begin{center}
\begin{tabular}{|c|c|c|}
\cline{2-3} 
\multicolumn{1}{c|}{} & $\left|\mathrm{Disc}_{+}\right|$ & $\left|\mathrm{Disc}_{-}\right|$\tabularnewline
\hline 
$\ff_{13^{2}}$ & $13$ & $13\cdot17\cdot61$\tabularnewline
\hline 
$\ff_{19^{2}}$ & $18691$ & $75011$\tabularnewline
\hline 
\end{tabular}
\par\end{center}

\noindent As these four discriminants are all different elements in
$\qq^{*}/(\qq^{*})^{2}$ we have $\mathrm{Disc}\left(\mathrm{NS}(X_{3,5,7,\ff_{13^{2}}})\right)\neq\mathrm{Disc}\left(\mathrm{NS}(X_{3,5,7,\ff_{19^{2}}})\right)$
and so a very general surface in $\mathcal{X}_{C,D,E}$ has Picard
number $17$. 
\item As $M$ has rank $16$, we know a very general surface $X$ of $\mathcal{X}$
has $\rho(X)\geq16$. We use the idea in Subsection~\ref{subsec:Reduction}
and find a surface whose reduction at a prime $p$ gives an upper
bound of $16$. We work over $\ff_{19}$ and let $X$ be the surface
defined by the point $[1,2,7,11,13]$. We calculate that the characteristic
polynomial of Frobenius acting on conics and lines is $f(x)=(x-1)^{8}(x+1)^{8}$,
hence we need to count points over $\ff_{19}$, $\ff_{19^{2}}$ and
$\ff_{19^{3}}$ to find the two possible characteristic polynomials
for $\Phi_{19}^{*}$. We find that, after rescaling 
\[
f_{19,+}(x)=\frac{1}{19}(x-1)^{8}(x+1)^{8}(19x^{6}+10x^{5}+29x^{4}+12x^{3}+29x^{2}+10x+19),
\]
 and a contradiction for 
\[
f_{19,-}(x)=\frac{1}{19}(x-1)^{7}(x+1)^{9}(19x^{6}-28x^{5}+47x^{4}-64x^{3}+47x^{2}-28x+19),
\]
as $f(x)\nmid f_{19,-}(x)$. Hence a very general surface in $\mathcal{X}$
has Picard number $16$. 
\end{itemize}
\end{proof}
Now that we know the rank of the Picard group of a very general member
of each family, we can prove the following proposition:
\begin{prop}
\label{prop:Generators}For a very general member of the families
$\mathcal{X},\mathcal{X}_{C,D,E},\mathcal{X}_{C,D},\mathcal{X}_{B}$
and $\mathcal{X}_{C}$, as well as the Fermat quartic, $F_{4}$, the
Picard groups are generated by lines and conics. 

In particular the matrices $M,M_{C,D,E},M_{C,D},M_{B}$ and $M_{C}$
as defined in Appendix\ref{sec:Matrices} define the Picard group
of a very general member of the families $\mathcal{X},\mathcal{X}_{C,D,E},\mathcal{X}_{C,D},\mathcal{X}_{B}$
and $\mathcal{X}_{C}$ respectively.
\end{prop}
\begin{proof}
First note that if $L_{1}\hookrightarrow L_{2}$ is primitive, then
no overlattice $L'$ of $L_{1}$ can be a sublattice of $L_{2}$.
Let $\mathcal{X}$ and $\mathcal{Y}$ are two families of K3 surfaces,
with $\mathcal{Y}$ a subfamily of $\mathcal{X}$. If $X$ and $Y$
denote a very general member of $\mathcal{X}$ and $\mathcal{Y}$,
then $\mathrm{Pic}(X)\hookrightarrow\mathrm{Pic}(Y)$ as the elements
of $\mathrm{Pic}(X)$ must specialise to elements of $\mathrm{Pic}(Y)$. 

With this in mind we start with the proven fact (see for example \cite{schutt2010lines})
that the Picard group of the Fermat quartic, denoted by $\mathrm{Pic}(F_{4})$,
is generated by lines. Upon calculating the Gram matrix of the $48$
lines on $F_{4}$, we find that the Picard group has discriminant
$-64$. On the other hand we calculate the following Gram matrix,
which we denote $M_{F_{4}}$, generated by $16$ conics and four lines
(each line coming from a different set of eight lines associated to
a point $q_{i}$):
\[
\left(\begin{array}{rrrrrrrrrrrrrrrrrrrr}
-2 & 0 & 0 & 2 & 2 & 0 & 0 & 1 & 2 & 1 & 0 & 0 & 0 & 1 & 2 & 0 & 1 & 0 & 1 & 0\\
0 & -2 & 2 & 2 & 0 & 0 & 0 & 1 & 1 & 1 & 0 & 0 & 1 & 0 & 1 & 1 & 1 & 1 & 1 & 0\\
0 & 2 & -2 & 0 & 2 & 2 & 1 & 2 & 1 & 1 & 0 & 0 & 1 & 2 & 1 & 1 & 1 & 0 & 1 & 0\\
2 & 2 & 0 & -2 & 2 & 2 & 2 & 1 & 0 & 0 & 1 & 1 & 1 & 2 & 1 & 1 & 1 & 0 & 0 & 1\\
2 & 0 & 2 & 2 & -2 & 2 & 0 & 2 & 1 & 2 & 1 & 1 & 1 & 0 & 1 & 1 & 0 & 1 & 1 & 0\\
0 & 0 & 2 & 2 & 2 & -2 & 1 & 0 & 2 & 0 & 1 & 1 & 1 & 0 & 1 & 1 & 1 & 1 & 0 & 1\\
0 & 0 & 1 & 2 & 0 & 1 & -2 & 1 & 2 & 1 & 0 & 0 & 1 & 0 & 2 & 1 & 0 & 0 & 1 & 0\\
1 & 1 & 2 & 1 & 2 & 0 & 1 & -2 & 1 & 1 & 2 & 2 & 0 & 0 & 2 & 1 & 1 & 0 & 0 & 1\\
2 & 1 & 1 & 0 & 1 & 2 & 2 & 1 & -2 & 2 & 1 & 2 & 1 & 1 & 1 & 1 & 0 & 1 & 1 & 1\\
1 & 1 & 1 & 0 & 2 & 0 & 1 & 1 & 2 & -2 & 0 & 0 & 2 & 2 & 1 & 1 & 1 & 0 & 0 & 0\\
0 & 0 & 0 & 1 & 1 & 1 & 0 & 2 & 1 & 0 & -2 & 0 & 2 & 2 & 1 & 1 & 1 & 0 & 1 & 0\\
0 & 0 & 0 & 1 & 1 & 1 & 0 & 2 & 2 & 0 & 0 & -2 & 1 & 1 & 0 & 2 & 1 & 0 & 1 & 0\\
0 & 1 & 1 & 1 & 1 & 1 & 1 & 0 & 1 & 2 & 2 & 1 & -2 & 0 & 2 & 0 & 1 & 0 & 1 & 1\\
1 & 0 & 2 & 2 & 0 & 0 & 0 & 0 & 1 & 2 & 2 & 1 & 0 & -2 & 2 & 2 & 0 & 1 & 1 & 1\\
2 & 1 & 1 & 1 & 1 & 1 & 2 & 2 & 1 & 1 & 1 & 0 & 2 & 2 & -2 & 2 & 1 & 1 & 0 & 1\\
0 & 1 & 1 & 1 & 1 & 1 & 1 & 1 & 1 & 1 & 1 & 2 & 0 & 2 & 2 & -2 & 1 & 0 & 0 & 0\\
1 & 1 & 1 & 1 & 0 & 1 & 0 & 1 & 0 & 1 & 1 & 1 & 1 & 0 & 1 & 1 & -2 & 1 & 1 & 0\\
0 & 1 & 0 & 0 & 1 & 1 & 0 & 0 & 1 & 0 & 0 & 0 & 0 & 1 & 1 & 0 & 1 & -2 & 0 & 0\\
1 & 1 & 1 & 0 & 1 & 0 & 1 & 0 & 1 & 0 & 1 & 1 & 1 & 1 & 0 & 0 & 1 & 0 & -2 & 1\\
0 & 0 & 0 & 1 & 0 & 1 & 0 & 1 & 1 & 0 & 0 & 0 & 1 & 1 & 1 & 0 & 0 & 0 & 1 & -2
\end{array}\right)
\]
$M_{F_{4}}$ has determinant $-64$ and hence does represent $\mathrm{Pic}(F_{4})$. 

Let $X_{C}$ be a very general surface in $\mathcal{X}_{C}$. We extracted
the matrix $M_{C}$, from the intersection matrix of the lines and
conics on $X_{C}$, by looking at the lines and conics which specialise
to a subset of the $16$ conics and four lines that lie on the Fermat
quartic (which makes sense since $X_{0}\in\mathcal{X}_{C}$ is the
Fermat quartic). We ended up with $16$ conics and three lines (which
must come from three different sets of eight lines) and hence we have
a rank $19$, i.e. full rank, sublattice of $\mathrm{Pic}(X_{C})$.
Notice that $M_{C}$ is a minor of $M_{F_{4}}$ (just remove the last
row and column), and the lines and conics generating $M_{C}$ specialise
to those generating the corresponding minor of $M_{F_{4}}$. Hence
the lattice defined by $M_{C}$ is a primitive sublattice of $\mathrm{Pic}(F_{4})$.
If $M_{C}$ did not define $\mathrm{Pic}(X_{C})$ then $\mathrm{Pic}(X_{C})$
would be an overlattice of $M_{C}$. Furthermore by the remark at
the beginning of the proof $\mathrm{Pic}(X_{C})$ would be a sublattice
of $\mathrm{Pic}(F_{4})$. This is a contradiction to the fact that
$M_{C}$ is already a primitive sublattice of $\mathrm{Pic}(F_{4})$.
Hence the lattice defined by $M_{C}$, which is generated by lines
and conics, is $\mathrm{Pic}(X_{C})$.

Similarly we extracted $M_{C,D}$ from the intersection matrix using
$M_{C}$ (and note it is a minor of $M_{C}$ by removing the last
row and column), $M_{C,D,E}$ using $M_{C,D}$ (a minor of $M_{C,D}$
by removing the last row and column) and $M$ using $M_{C,D,E}$.
Hence by the same argument, they represent respectively $\mathrm{Pic}(X_{C,D})$,
$\mathrm{Pic}(X_{C,D,E})$ and $\mathrm{Pic}(X)$. Finally, we extracted
$M_{B}$ from $F_{4}$ using the same process (and notice it is a
minor of $F_{4}$ by removing column and row $18$), finishing the
proof.
\end{proof}
\begin{notation*}
Let $M,N$ be matrices, then we use $M^{N}$ to mean $N^{T}\cdot M\cdot N$.
\end{notation*}
We now have all the tools to prove our main result
\begin{thm}
\label{thm:Main Result Final}Let $p=[A,B,C,D,E]\in\pp^{4}$ define
the quartic $X_{p}:A(x^{4}+y^{4}+z^{4}+w^{4})+Bxyzw+C(x^{2}y^{2}+z^{2}w^{2})+D(x^{2}z^{2}+y^{2}w^{2})+E(x^{2}w^{2}+y^{2}z^{2})\subset\pp^{3}$.
Then

\begin{itemize}
\item A very general member of family parameterised by $\pp^{4}$ contains
no lines, has Picard rank $16$ and Picard group isomorphic to \textup{$E_{8}\left\langle -1\right\rangle \oplus U\oplus D_{5}\left\langle -2\right\rangle \oplus\left\langle -4\right\rangle ,$}
\item A very general member of the family parameterised by $[A,(DE-2AC)/A,C,D,E]$
contains exactly eight lines, has rank $17$ and Picard group isomorphic
to $E_{8}\left\langle -1\right\rangle \oplus U\oplus A_{2}\left\langle -2\right\rangle \oplus\left(D_{4}\left\langle -1\right\rangle \oplus\left\langle -2\right\rangle \right)^{N}$,
with 
\[
N=\begin{pmatrix}1 & 0 & 0 & 0 & 0\\
0 & 1 & 0 & 0 & 0\\
0 & 0 & 1 & 0 & 0\\
0 & 0 & 0 & 1 & 0\\
0 & 0 & 1 & 0 & 2
\end{pmatrix},
\]
\item A very general member of the family parameterised by $[A,0,C,D,2AC/D]$
contains exactly $16$ lines, has rank $18$ and Picard group isomorphic
to $E_{8}\left\langle -1\right\rangle \oplus U\oplus A_{7}\left\langle -1\right\rangle ^{I_{4,2}}\oplus\left\langle -8\right\rangle $,
with $I_{4,2}=\mathrm{Diag}([1,1,1,2,1,1,1])$,
\item A very general member of the family parameterised by $[A,B(2A-B)/A,B,B,B]$
contains exactly $24$ lines, has rank $19$ and Picard group isomorphic
to $E_{8}\left\langle -1\right\rangle \oplus U\oplus D_{8}\left\langle -1\right\rangle \oplus\left\langle -40\right\rangle $,
\item A very general member of the family parameterised by $[A,0,C,0,0]$
contains exactly $32$ lines, has rank $19$ and Picard group isomorphic
to $E_{8}\left\langle -1\right\rangle \oplus U\oplus D_{8}\left\langle -1\right\rangle ^{I_{8,2}}\oplus\left\langle -8\right\rangle $
with $I_{8,2}=\mathrm{Diag}([1,1,1,1,1,1,1,2])$,
\item The surface defined by the point $[\sqrt{-3},12(\sqrt{-3}-1),6,6,-6]$
contains exactly $32$ lines, has rank $20$ and Picard group isomorphic
either to $E_{8}\left\langle -1\right\rangle ^{\oplus2}\oplus U\oplus\left\langle -4\right\rangle \oplus\left\langle -24\right\rangle $
or to $E_{8}\left\langle -1\right\rangle ^{\oplus2}\oplus U\oplus\left\langle -4\right\rangle \oplus\left\langle -6\right\rangle $
(but not both),
\item The Fermat quartic defined by the point $[1,0,0,0,0]$ contains exactly
$48$ lines, has rank $20$ and Picard group isomorphic to $E_{8}\left\langle -1\right\rangle ^{\oplus2}\oplus U\oplus\left\langle -8\right\rangle \oplus\left\langle -8\right\rangle $.
\end{itemize}
Possibly except for the point $[\sqrt{-3},12(\sqrt{-3}-1),6,6,-6]$,
the Picard group is generated by the lines and conics lying on the
surface.
\end{thm}
\begin{proof}
The claim about the number of lines each very general member contains
is proven in Lemma \ref{lem:Number of lines} while the rank is proven
in Proposition \ref{prop:Picard Rank}. Apart from the surface defined
by $[\sqrt{-3},12(\sqrt{-3}-1),6,6,-6]$, the claim that the Picard
group of a very general member is defined by lines and conics is proven
by Proposition \ref{prop:Generators}.

As all the Picard groups are even and indefinite, and in each case
the rank is large enough, we can apply Theorem \ref{thm:Nikulin-Iso}
to each of our Picard groups. Specifically one can check (see \cite{ThesisCode})
that the discriminant form, rank and signature of the lattices defined
by $M,M_{C,D,E},M_{C,D},M_{B},M_{C}$ and $M_{F_{4}}$ are the same
as the discriminant form, rank and signature of the lattice

\begin{itemize}
\item $E_{8}\left\langle -1\right\rangle \oplus U\oplus D_{5}\left\langle -2\right\rangle \oplus\left\langle -4\right\rangle ,$
\item $E_{8}\left\langle -1\right\rangle \oplus U\oplus A_{2}\left\langle -2\right\rangle \oplus\left(D_{4}\left\langle -1\right\rangle \oplus\left\langle -2\right\rangle \right)^{N}$, 
\item $E_{8}\left\langle -1\right\rangle \oplus U\oplus A_{7}\left\langle -1\right\rangle ^{I_{4,2}}\oplus\left\langle -8\right\rangle $,
\item $E_{8}\left\langle -1\right\rangle \oplus U\oplus D_{8}\left\langle -1\right\rangle \oplus\left\langle -40\right\rangle $, 
\item $E_{8}\left\langle -1\right\rangle \oplus U\oplus D_{8}\left\langle -1\right\rangle ^{I_{8,2}}\oplus\left\langle -8\right\rangle $,
\item and $E_{8}\left\langle -1\right\rangle ^{\oplus2}\oplus U\oplus\left\langle -8\right\rangle \oplus\left\langle -8\right\rangle $
respectively.
\end{itemize}
For the surface $Y$, defined by $[\sqrt{-3},12(\sqrt{-3}-1),6,6,-6]$,
the lattice defined by $M_{Y}$ is isomorphic to $E_{8}\left\langle -1\right\rangle ^{\oplus2}\oplus U\oplus\left\langle -4\right\rangle \oplus\left\langle -24\right\rangle $.
While we don't know that the lattice defined by $M_{Y}$ is the Picard
group of $Y$, we know that it is a full rank sublattice of it. One
can then use Theorem \ref{thm:Niikulin-overlattice} to find all overlattices
of it, of which there is only one, and use Theorem \ref{thm:Nikulin-Iso}
to identify said lattice using its discriminant form, rank and signature.
\end{proof}
Recall that at the end of Section \ref{sec:The-Families-and-Lines},
we had a diagram illustrating the various subfamilies of $\mathcal{X}$
containing lines and how they fitted together. Here we reproduce the
same diagram where instead of the families, we put together the Picard
group of the generic member of each family (except for the surface
$Y$, where we put the two possible Picard groups), and instead of
the dimension of each family we put the rank of the Picard group.

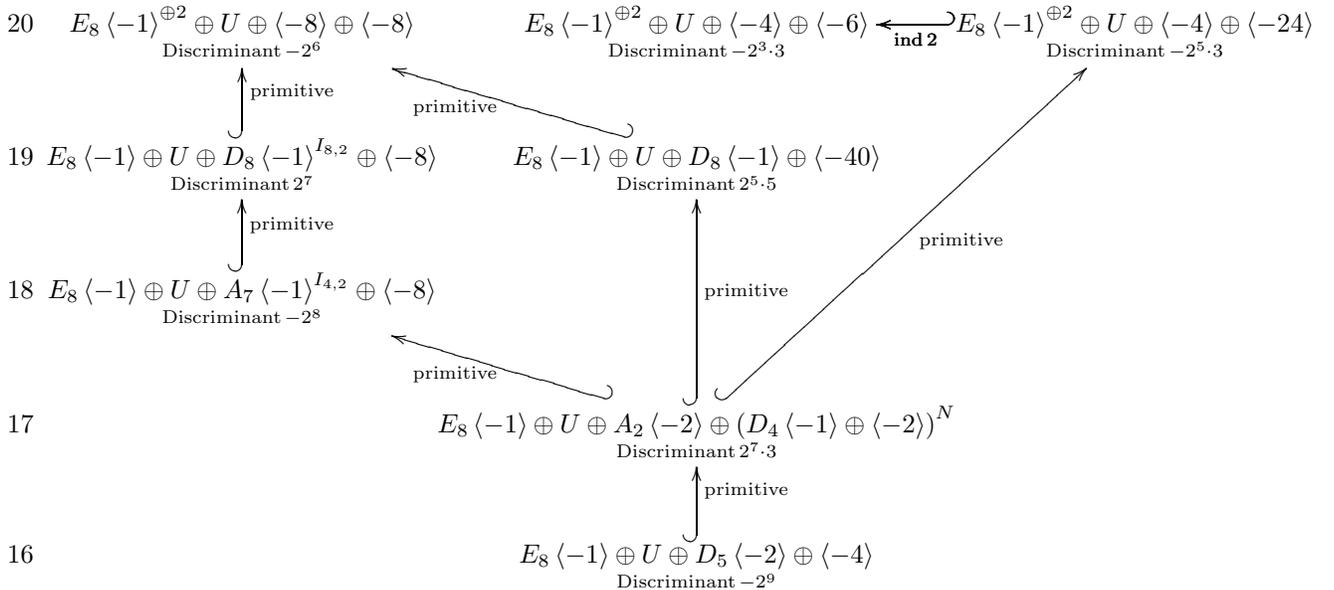
\selectlanguage{english}%
\[
\xyR{2pc}\xyC{-0.5pc}\xymatrix{\mathrm{rank}\\
20 & \underset{\mathrm{Discriminant}\,-2^{6}}{E_{8}\left\langle -1\right\rangle ^{\oplus2}\oplus U\oplus\left\langle -8\right\rangle \oplus\left\langle -8\right\rangle } & \underset{\mathrm{Discriminant}\,-2^{3}\cdot3}{E_{8}\left\langle -1\right\rangle ^{\oplus2}\oplus U\oplus\left\langle -4\right\rangle \oplus\left\langle -6\right\rangle } & \underset{\mathrm{Discriminant\,-2^{5}\cdot3}}{E_{8}\left\langle -1\right\rangle ^{\oplus2}\oplus U\oplus\left\langle -4\right\rangle \oplus\left\langle -24\right\rangle \ar@{_{(}->}[l]^{\mathrm{ind\,}2}}\\
19 & \underset{\mathrm{Discriminant\,}2^{7}}{E_{8}\left\langle -1\right\rangle \oplus U\oplus D_{8}\left\langle -1\right\rangle ^{I_{8,2}}\oplus\left\langle -8\right\rangle }\ar@{^{(}->}[u]_{\mathrm{primitive}} & \underset{\mathrm{Discriminant\,}2^{5}\cdot5}{E_{8}\left\langle -1\right\rangle \oplus U\oplus D_{8}\left\langle -1\right\rangle \oplus\left\langle -40\right\rangle }\ar@{_{(}->}[ul]^{\mathrm{primitive}}\\
18 & \underset{\mathrm{Discriminant\,}-2^{8}}{E_{8}\left\langle -1\right\rangle \oplus U\oplus A_{7}\left\langle -1\right\rangle ^{I_{4,2}}\oplus\left\langle -8\right\rangle }\ar@{^{(}->}[u]_{\mathrm{primitive}}\\
17 &  & \underset{\mathrm{Discriminant\,}2^{7}\cdot3}{E_{8}\left\langle -1\right\rangle \oplus U\oplus A_{2}\left\langle -2\right\rangle \oplus\left(D_{4}\left\langle -1\right\rangle \oplus\left\langle -2\right\rangle \right)^{N}}\ar@{^{(}->}[uu]_{\mathrm{primitive}}\ar@{_{(}->}[ul]^{\mathrm{primitive}}\ar@{^{(}->}[uuur]_{\mathrm{primitive}}\\
16 &  & \underset{\mathrm{Discriminant\,}-2^{9}}{E_{8}\left\langle -1\right\rangle \oplus U\oplus D_{5}\left\langle -2\right\rangle \oplus\left\langle -4\right\rangle }\ar@{^{(}->}[u]_{\mathrm{primitive}}
}
\]

\selectlanguage{british}%

\subsection{Method}

We include here two examples of how the isomorphic lattices were found
for Theorem \ref{thm:Main Result Final}, which the reader might find
useful. Those two examples illustrate the two different approaches
we took in identifying the lattices.

We start with the lattice defined by $M$, i.e. the Picard group of
a very general member $X$ of $\mathcal{X}$. We know that $M$ has
signature $(1,15)$ and rank $16$. We calculate its discriminant
group to be $C_{2}^{4}\times C_{4}\times C_{8}$, and $\mathrm{Pic}(X)$
has discriminant $-512$ (this concurs with the proof of \cite[Thm 7.3, Cor 7.4]{eklund2010}).
By Theorem \ref{thm:Nikulin-E8-U}, we see that we can fit in one
copy of $E_{8}(-1)$ and one copy of $U$ in $\mathrm{Pic}(X)$, i.e.,
$\mathrm{Pic}(X)=U\oplus E_{8}\left\langle -1\right\rangle \oplus T$
where $T$ is a lattice with the same discriminant group and discriminant
form as $\mathrm{Pic}(X)$, but with signature $(0,6)$.

Recall that $A_{L}$ denotes the discriminant group of a lattice $L$,
and $q_{L}$ its discriminant form. We calculate the discriminant
form and find that:
\begin{itemize}
\item If $x\in A_{\mathrm{Pic}(X)}$ has order $2$ then $q_{\mathrm{Pic}(X)}(x)\in\{0,1\}$,
\item If $x\in A_{\mathrm{Pic}(X)}$ has order $4$ then $q_{\mathrm{Pic}(X)}(x)\in\{-\frac{3}{4},-\frac{2}{4},-\frac{1}{4},\frac{1}{4},\frac{2}{4},\frac{3}{4}\}$,
\item If $x\in A_{\mathrm{Pic}(X)}$ has order $8$ then $q_{\mathrm{Pic}(X)}(x)\in\{-\frac{7}{8},-\frac{5}{8},\frac{1}{8},\frac{3}{8}\}$.
\end{itemize}
As the lattice $\left\langle -4\right\rangle $ has discriminant form
$-\frac{1}{4}$ and discriminant group $C_{4}$, we guess that it
appears as one of the summands of $\mathrm{Pic}(X)$. Using Table
\ref{tab:Lattices info} and the fact we need negative definite lattices,
we see that the $C_{8}$ factor could arise from $A_{7}\left\langle -1\right\rangle $,
$A_{3}\left\langle -2\right\rangle $, $\left\langle -8\right\rangle $
or $D_{2n+1}\left\langle -2\right\rangle $. As $A_{7}\left\langle -1\right\rangle $
has too large of a rank (greater than six), and both $A_{3}\left\langle -2\right\rangle $
and $\left\langle -8\right\rangle $ have an element of order $8$
with discriminant form $-\frac{3}{8}$ and $-\frac{1}{8}$ respectively,
they can not be a factor of $\mathrm{Pic}(X)$. On the other hand,
$D_{5}\left\langle -2\right\rangle $ does not give any obvious contradiction
while having discriminant group $C_{2}^{4}\times C_{8}$. We guess
that it is a factor of $\mathrm{Pic}(X)$. Hence putting everything
together we check that $\mathrm{Pic}(X)\cong U\otimes E_{8}\left\langle -1\right\rangle \otimes D_{5}\left\langle -2\right\rangle \otimes\left\langle -4\right\rangle $.
It is easy to see they have the same rank and signature; and a calculation
checks they have the same discriminant form, namely both discriminant
group have a basis $\{g_{1},\dots,g_{6}\}$ such that the discriminant
form is given by
\begin{eqnarray*}
M_{q_{L}}(a_{ij}) & = & \begin{cases}
q_{L}(g_{i}+g_{j}) & i\neq j\\
q_{L}(g_{i}) & i=j
\end{cases}\\
 & = & \left(\begin{array}{cccccc}
0 & 1 & 1 & 1 & -\frac{1}{4} & -\frac{5}{8}\\
1 & 1 & 1 & 0 & \frac{3}{4} & -\frac{5}{8}\\
1 & 1 & 1 & 0 & \frac{3}{4} & -\frac{5}{8}\\
1 & 0 & 0 & 0 & -\frac{1}{4} & \frac{3}{8}\\
-\frac{1}{4} & \frac{3}{4} & \frac{3}{4} & -\frac{1}{4} & -\frac{1}{4} & -\frac{7}{8}\\
-\frac{5}{8} & -\frac{5}{8} & -\frac{5}{8} & \frac{3}{8} & -\frac{7}{8} & -\frac{5}{8}
\end{array}\right).
\end{eqnarray*}

Our second example is with the lattice defined by $M_{C}$, i.e. the
Picard group of a very general member $X_{C}$ of $\mathcal{X}_{C}$.
We know that $M_{C}$ has signature $(1,18)$ and rank $19$. We calculate
that $\mathrm{Pic}(X_{C})$ has discriminant $128$ and discriminant
group $C_{4}^{2}\times C_{8}$. By Theorem \ref{thm:Nikulin-E8-U},
we know that $\mathrm{Pic}(X_{C})\cong E_{8}\left\langle -1\right\rangle \oplus U\oplus T$,
where $T$ is a lattice of signature $(0,9)$ with discriminant group
$C_{4}^{2}\times C_{8}$ and discriminant form as:
\begin{itemize}
\item If $x\in A_{\mathrm{Pic}(X_{C})}$ has order $2$ then $q_{\mathrm{Pic}(X_{C})}(x)\in\{0\}$,
\item If $x\in A_{\mathrm{Pic}(X_{C})}$ has order $4$ then $q_{\mathrm{Pic}(X_{C})}(x)\in\{-\frac{1}{2},0,\frac{1}{2},1\}$,
\item If $x\in A_{\mathrm{Pic}(X_{C})}$ has order $8$ then $q_{\mathrm{Pic}(X_{C})}(x)\in\{-\frac{5}{8},-\frac{1}{8},\frac{3}{8},\frac{7}{8}\}$.
\end{itemize}
As there is no negative definite lattice in Table \ref{tab:Lattices info}
which gives a copy of $C_{4}$ without giving an element of discriminant
form $\frac{2n+1}{4}$ for some $n$, we deduce that $T$ can not
be written simply in terms of scaled root lattices. Instead we use
Theorem \ref{thm:Niikulin-overlattice} to find an overlattice of
$\mathrm{Pic}(X_{C})$ that we can identify. In particular, if we
let $\left\{ e_{i}\right\} $ be the basis given by $M_{C}$, then
$\frac{1}{2}(e_{4}+e_{5}+e_{10}+e_{11}+e_{13}+e_{14})\in A_{\mathrm{Pic}(X_{C})}$
has order two and discriminant form zero. This generates an isotropic
subgroup of $A_{\mathrm{Pic}(X_{C})}$ and gives a corresponding index
two overlattice. This overlattice, $L$, has discriminant group $C_{2}^{2}\times C_{8}$
and discriminant form given as:
\begin{itemize}
\item If $x\in A_{L}$ has order $2$ then $q_{L}(x)\in\{-\frac{1}{2},0,\frac{1}{2}\}$,
\item If $x\in A_{L}$ has order $4$ then $q_{L}(x)\in\{-\frac{1}{2},0,\frac{1}{2}\}$,
\item If $x\in A_{L}$ has order $8$ then $q_{L}(x)\in\{-\frac{5}{8},-\frac{1}{8},\frac{3}{8},\frac{7}{8}\}$.
\end{itemize}
Following our first example this allows us to guess that $L\cong E_{8}\left\langle -1\right\rangle \oplus U\oplus D_{8}\left\langle -1\right\rangle \oplus\left\langle -8\right\rangle $.
We check that is the case, as they both have rank $19$, signature
$(1,18)$ and discriminant form given by
\begin{eqnarray*}
M_{q_{L}}(a_{ij}) & = & \begin{cases}
q_{L}(g_{i}+g_{j}) & i\neq j\\
q_{L}(g_{i}) & i=j
\end{cases}\\
 & = & \begin{pmatrix}0 & 1 & -\frac{1}{8}\\
1 & 0 & \frac{7}{8}\\
-\frac{1}{8} & \frac{7}{8} & \frac{7}{8}
\end{pmatrix}.
\end{eqnarray*}
Knowing that $\mathrm{Pic}(X)$ is an index two full rank sublattice
of $E_{8}\left\langle -1\right\rangle \oplus U\oplus D_{8}\left\langle -1\right\rangle \oplus\left\langle -8\right\rangle $,
we enumerate the index two full rank sublattices of $E_{8}\left\langle -1\right\rangle \oplus U\oplus D_{8}\left\langle -1\right\rangle \oplus\left\langle -8\right\rangle $
until we find one that has the same discriminant form as $\mathrm{Pic}(X)$. 

\newpage{}

\appendix

\section{The equations of the lines \label{sec:Eqn of lines}}

The following table gives the equations of the $8$ lines lying on
the point $p=[A,B,C,D,E]\in\pp^{5}$ depending on which tangent cone
its lies on.

\begin{table}[h]
\selectlanguage{english}%
\begin{tabular}{|>{\centering}p{2cm}|c|>{\centering}p{11.5cm}|}
\hline 
\textbf{Tangent cone to the point} & \textbf{Conics associated} & \textbf{Lines}\tabularnewline
\hline 
\hline 
\foreignlanguage{british}{$q_{1}$} & $x^{2}-y^{2}-z^{2}+w^{2}$ & $2\sqrt{q_{+C}}x+\sqrt{p_{-1}}z+\sqrt{-p_{+0}}w=2\sqrt{q_{+C}}y+\sqrt{-p_{+0}}z+\sqrt{p_{-1}}w=0$
$2\sqrt{q_{+C}}x+\sqrt{p_{+1}}z+\sqrt{-p_{-0}}w=2\sqrt{q_{+C}}y-\sqrt{-p_{-0}}z-\sqrt{p_{+1}}w=0$\tabularnewline
\hline 
\foreignlanguage{british}{$q_{2}$} & $x^{2}-y^{2}+z^{2}-w^{2}$ & $2\sqrt{q_{+C}}x+\sqrt{-p_{+0}}z+\sqrt{p_{-1}}w=2\sqrt{q_{+C}}y+\sqrt{p_{-1}}z+\sqrt{-p_{+0}}w=0$
$2\sqrt{q_{+C}}x+\sqrt{-p_{-0}}z+\sqrt{p_{+1}}w=2\sqrt{q_{+C}}y-\sqrt{p_{+1}}z-\sqrt{-p_{-0}}w=0$\tabularnewline
\hline 
\foreignlanguage{british}{$q_{3}$} & $x^{2}+y^{2}-z^{2}-w^{2}$ & $2\sqrt{q_{-C}}x+\sqrt{p_{+3}}z+\sqrt{p_{-2}}w=2\sqrt{q_{-C}}y-\sqrt{p_{-2}}z+\sqrt{p_{+3}}w=0$
$2\sqrt{q_{-C}}x+\sqrt{p_{-3}}z+\sqrt{p_{+2}}w=2\sqrt{q_{-C}}y+\sqrt{p_{+2}}z-\sqrt{p_{-3}}w=0$\tabularnewline
\hline 
\foreignlanguage{british}{$q_{4}$} & $x^{2}+y^{2}+z^{2}+w^{2}$ & $2\sqrt{q_{-C}}x+\sqrt{-p_{-2}}z+\sqrt{-p_{+3}}w=2\sqrt{q_{-C}}y-\sqrt{-p_{+3}}z+\sqrt{-p_{-2}}w=0$
$2\sqrt{q_{-C}}x+\sqrt{-p_{+2}}z+\sqrt{-p_{-3}}w=2\sqrt{q_{-C}}y+\sqrt{-p_{-3}}z-\sqrt{-p_{+2}}w=0$\tabularnewline
\hline 
\foreignlanguage{british}{$q_{5}$} & $xy-zw$ & $2\sqrt{A}x+\left(\sqrt{q_{-D}}+\sqrt{-q_{+D}}\right)z=2\sqrt{A}y+\left(\sqrt{q_{-D}}-\sqrt{-q_{+D}}\right)w=0$

$2\sqrt{A}x+\left(\sqrt{q_{-E}}+\sqrt{-q_{+E}}\right)w=2\sqrt{A}y+\left(\sqrt{q_{-E}}-\sqrt{-q_{+E}}\right)z=0$\tabularnewline
\hline 
\foreignlanguage{british}{$q_{6}$} & $xy+zw$ & $2\sqrt{A}x+\left(\sqrt{q_{-D}}+\sqrt{-q_{+D}}\right)z=2\sqrt{A}y-\left(\sqrt{q_{-D}}-\sqrt{-q_{+D}}\right)w=0$$2\sqrt{A}x+\left(\sqrt{q_{-E}}+\sqrt{-q_{+E}}\right)w=2\sqrt{A}y-\left(\sqrt{q_{-E}}-\sqrt{-q_{+E}}\right)z=0$\tabularnewline
\hline 
\foreignlanguage{british}{$q_{7}$} & $xz-yw$ & $2\sqrt{A}x+\left(\sqrt{q_{-C}}+\sqrt{-q_{+C}}\right)y=2\sqrt{A}z+\left(\sqrt{q_{-C}}-\sqrt{-q_{+C}}\right)w=0$

$2\sqrt{A}x+\left(\sqrt{q_{-E}}+\sqrt{-q_{+E}}\right)w=2\sqrt{A}z+\left(\sqrt{q_{-E}}-\sqrt{-q_{+E}}\right)zy=0$\tabularnewline
\hline 
\foreignlanguage{british}{$q_{8}$} & $xz+yw$ & $2\sqrt{A}x+\left(\sqrt{q_{-C}}+\sqrt{-q_{+C}}\right)y=2\sqrt{A}z-\left(\sqrt{q_{-C}}-\sqrt{-q_{+C}}\right)w=0$$2\sqrt{A}x+\left(\sqrt{q_{-E}}+\sqrt{-q_{+E}}\right)w=2\sqrt{A}z-\left(\sqrt{q_{-E}}-\sqrt{-q_{+E}}\right)y=0$\tabularnewline
\hline 
\foreignlanguage{british}{$q_{9}$} & $xw-yz$ & $2\sqrt{A}x+\left(\sqrt{q_{-C}}+\sqrt{-q_{+C}}\right)y=2\sqrt{A}w+\left(\sqrt{q_{-C}}-\sqrt{-q_{+C}}\right)z=0$

$2\sqrt{A}x+\left(\sqrt{q_{-D}}+\sqrt{-q_{+D}}\right)z=2\sqrt{A}w+\left(\sqrt{q_{-D}}-\sqrt{-q_{+D}}\right)y=0$\tabularnewline
\hline 
\foreignlanguage{british}{$q_{10}$} & $xw+yz$ & $2\sqrt{A}x+\left(\sqrt{q_{-C}}+\sqrt{-q_{+C}}\right)y=2\sqrt{A}w-\left(\sqrt{q_{-C}}-\sqrt{-q_{+C}}\right)z=0$

$2\sqrt{A}x+\left(\sqrt{q_{-D}}+\sqrt{-q_{+D}}\right)z=2\sqrt{A}w-\left(\sqrt{q_{-D}}-\sqrt{-q_{+D}}\right)y=0$\tabularnewline
\hline 
\end{tabular}

\selectlanguage{british}%
\caption{Equations of lines \label{tab:Equations-of-lines}}

\end{table}

\newpage{}

\section{List Of Gram Matrices \label{sec:Matrices}}

\vfill{}

For a very general member of the family $\mathcal{X}$, a full rank
minor of minimal discriminant is

\selectlanguage{english}%
\[
M=\left(\begin{array}{rrrrrrrrrrrrrrrr}
-2 & 0 & 0 & 2 & 2 & 0 & 0 & 1 & 2 & 1 & 0 & 0 & 0 & 1 & 2 & 0\\
0 & -2 & 2 & 2 & 0 & 0 & 0 & 1 & 1 & 1 & 0 & 0 & 1 & 0 & 1 & 1\\
0 & 2 & -2 & 0 & 2 & 2 & 1 & 2 & 1 & 1 & 0 & 0 & 1 & 2 & 1 & 1\\
2 & 2 & 0 & -2 & 2 & 2 & 2 & 1 & 0 & 0 & 1 & 1 & 1 & 2 & 1 & 1\\
2 & 0 & 2 & 2 & -2 & 2 & 0 & 2 & 1 & 2 & 1 & 1 & 1 & 0 & 1 & 1\\
0 & 0 & 2 & 2 & 2 & -2 & 1 & 0 & 2 & 0 & 1 & 1 & 1 & 0 & 1 & 1\\
0 & 0 & 1 & 2 & 0 & 1 & -2 & 1 & 2 & 1 & 0 & 0 & 1 & 0 & 2 & 1\\
1 & 1 & 2 & 1 & 2 & 0 & 1 & -2 & 1 & 1 & 2 & 2 & 0 & 0 & 2 & 1\\
2 & 1 & 1 & 0 & 1 & 2 & 2 & 1 & -2 & 2 & 1 & 2 & 1 & 1 & 1 & 1\\
1 & 1 & 1 & 0 & 2 & 0 & 1 & 1 & 2 & -2 & 0 & 0 & 2 & 2 & 1 & 1\\
0 & 0 & 0 & 1 & 1 & 1 & 0 & 2 & 1 & 0 & -2 & 0 & 2 & 2 & 1 & 1\\
0 & 0 & 0 & 1 & 1 & 1 & 0 & 2 & 2 & 0 & 0 & -2 & 1 & 1 & 0 & 2\\
0 & 1 & 1 & 1 & 1 & 1 & 1 & 0 & 1 & 2 & 2 & 1 & -2 & 0 & 2 & 0\\
1 & 0 & 2 & 2 & 0 & 0 & 0 & 0 & 1 & 2 & 2 & 1 & 0 & -2 & 2 & 2\\
2 & 1 & 1 & 1 & 1 & 1 & 2 & 2 & 1 & 1 & 1 & 0 & 2 & 2 & -2 & 2\\
0 & 1 & 1 & 1 & 1 & 1 & 1 & 1 & 1 & 1 & 1 & 2 & 0 & 2 & 2 & -2
\end{array}\right)
\]

\selectlanguage{british}%
\vfill{}

For a very general member of the family $\mathcal{X}_{C,D,E}$, a
full rank minor of minimal discriminant is

\[
M_{C,D,E}=\left(\begin{array}{rrrrrrrrrrrrrrrrr}
-2 & 0 & 0 & 2 & 2 & 0 & 0 & 1 & 2 & 1 & 0 & 0 & 0 & 1 & 2 & 0 & 1\\
0 & -2 & 2 & 2 & 0 & 0 & 0 & 1 & 1 & 1 & 0 & 0 & 1 & 0 & 1 & 1 & 1\\
0 & 2 & -2 & 0 & 2 & 2 & 1 & 2 & 1 & 1 & 0 & 0 & 1 & 2 & 1 & 1 & 1\\
2 & 2 & 0 & -2 & 2 & 2 & 2 & 1 & 0 & 0 & 1 & 1 & 1 & 2 & 1 & 1 & 1\\
2 & 0 & 2 & 2 & -2 & 2 & 0 & 2 & 1 & 2 & 1 & 1 & 1 & 0 & 1 & 1 & 0\\
0 & 0 & 2 & 2 & 2 & -2 & 1 & 0 & 2 & 0 & 1 & 1 & 1 & 0 & 1 & 1 & 1\\
0 & 0 & 1 & 2 & 0 & 1 & -2 & 1 & 2 & 1 & 0 & 0 & 1 & 0 & 2 & 1 & 0\\
1 & 1 & 2 & 1 & 2 & 0 & 1 & -2 & 1 & 1 & 2 & 2 & 0 & 0 & 2 & 1 & 1\\
2 & 1 & 1 & 0 & 1 & 2 & 2 & 1 & -2 & 2 & 1 & 2 & 1 & 1 & 1 & 1 & 0\\
1 & 1 & 1 & 0 & 2 & 0 & 1 & 1 & 2 & -2 & 0 & 0 & 2 & 2 & 1 & 1 & 1\\
0 & 0 & 0 & 1 & 1 & 1 & 0 & 2 & 1 & 0 & -2 & 0 & 2 & 2 & 1 & 1 & 1\\
0 & 0 & 0 & 1 & 1 & 1 & 0 & 2 & 2 & 0 & 0 & -2 & 1 & 1 & 0 & 2 & 1\\
0 & 1 & 1 & 1 & 1 & 1 & 1 & 0 & 1 & 2 & 2 & 1 & -2 & 0 & 2 & 0 & 1\\
1 & 0 & 2 & 2 & 0 & 0 & 0 & 0 & 1 & 2 & 2 & 1 & 0 & -2 & 2 & 2 & 0\\
2 & 1 & 1 & 1 & 1 & 1 & 2 & 2 & 1 & 1 & 1 & 0 & 2 & 2 & -2 & 2 & 1\\
0 & 1 & 1 & 1 & 1 & 1 & 1 & 1 & 1 & 1 & 1 & 2 & 0 & 2 & 2 & -2 & 1\\
1 & 1 & 1 & 1 & 0 & 1 & 0 & 1 & 0 & 1 & 1 & 1 & 1 & 0 & 1 & 1 & -2
\end{array}\right)
\]

\vfill{}

\newpage{}

For a very general member of the family $\mathcal{X}_{C,D}$, a full
rank minor of minimal discriminant is

\[
M_{C,D}=\left(\begin{array}{rrrrrrrrrrrrrrrrrr}
-2 & 0 & 0 & 2 & 2 & 0 & 0 & 1 & 2 & 1 & 0 & 0 & 0 & 1 & 2 & 0 & 1 & 0\\
0 & -2 & 2 & 2 & 0 & 0 & 0 & 1 & 1 & 1 & 0 & 0 & 1 & 0 & 1 & 1 & 1 & 1\\
0 & 2 & -2 & 0 & 2 & 2 & 1 & 2 & 1 & 1 & 0 & 0 & 1 & 2 & 1 & 1 & 1 & 0\\
2 & 2 & 0 & -2 & 2 & 2 & 2 & 1 & 0 & 0 & 1 & 1 & 1 & 2 & 1 & 1 & 1 & 0\\
2 & 0 & 2 & 2 & -2 & 2 & 0 & 2 & 1 & 2 & 1 & 1 & 1 & 0 & 1 & 1 & 0 & 1\\
0 & 0 & 2 & 2 & 2 & -2 & 1 & 0 & 2 & 0 & 1 & 1 & 1 & 0 & 1 & 1 & 1 & 1\\
0 & 0 & 1 & 2 & 0 & 1 & -2 & 1 & 2 & 1 & 0 & 0 & 1 & 0 & 2 & 1 & 0 & 0\\
1 & 1 & 2 & 1 & 2 & 0 & 1 & -2 & 1 & 1 & 2 & 2 & 0 & 0 & 2 & 1 & 1 & 0\\
2 & 1 & 1 & 0 & 1 & 2 & 2 & 1 & -2 & 2 & 1 & 2 & 1 & 1 & 1 & 1 & 0 & 1\\
1 & 1 & 1 & 0 & 2 & 0 & 1 & 1 & 2 & -2 & 0 & 0 & 2 & 2 & 1 & 1 & 1 & 0\\
0 & 0 & 0 & 1 & 1 & 1 & 0 & 2 & 1 & 0 & -2 & 0 & 2 & 2 & 1 & 1 & 1 & 0\\
0 & 0 & 0 & 1 & 1 & 1 & 0 & 2 & 2 & 0 & 0 & -2 & 1 & 1 & 0 & 2 & 1 & 0\\
0 & 1 & 1 & 1 & 1 & 1 & 1 & 0 & 1 & 2 & 2 & 1 & -2 & 0 & 2 & 0 & 1 & 0\\
1 & 0 & 2 & 2 & 0 & 0 & 0 & 0 & 1 & 2 & 2 & 1 & 0 & -2 & 2 & 2 & 0 & 1\\
2 & 1 & 1 & 1 & 1 & 1 & 2 & 2 & 1 & 1 & 1 & 0 & 2 & 2 & -2 & 2 & 1 & 1\\
0 & 1 & 1 & 1 & 1 & 1 & 1 & 1 & 1 & 1 & 1 & 2 & 0 & 2 & 2 & -2 & 1 & 0\\
1 & 1 & 1 & 1 & 0 & 1 & 0 & 1 & 0 & 1 & 1 & 1 & 1 & 0 & 1 & 1 & -2 & 1\\
0 & 1 & 0 & 0 & 1 & 1 & 0 & 0 & 1 & 0 & 0 & 0 & 0 & 1 & 1 & 0 & 1 & -2
\end{array}\right)
\]

\vfill{}

For a very general member of the family $\mathcal{X}_{B}$, a full
rank minor of minimal discriminant is

\[
M_{B}=\left(\begin{array}{rrrrrrrrrrrrrrrrrrr}
-2 & 0 & 0 & 2 & 2 & 0 & 0 & 1 & 2 & 1 & 0 & 0 & 0 & 1 & 2 & 0 & 1 & 1 & 0\\
0 & -2 & 2 & 2 & 0 & 0 & 0 & 1 & 1 & 1 & 0 & 0 & 1 & 0 & 1 & 1 & 1 & 1 & 0\\
0 & 2 & -2 & 0 & 2 & 2 & 1 & 2 & 1 & 1 & 0 & 0 & 1 & 2 & 1 & 1 & 1 & 1 & 0\\
2 & 2 & 0 & -2 & 2 & 2 & 2 & 1 & 0 & 0 & 1 & 1 & 1 & 2 & 1 & 1 & 1 & 0 & 1\\
2 & 0 & 2 & 2 & -2 & 2 & 0 & 2 & 1 & 2 & 1 & 1 & 1 & 0 & 1 & 1 & 0 & 1 & 0\\
0 & 0 & 2 & 2 & 2 & -2 & 1 & 0 & 2 & 0 & 1 & 1 & 1 & 0 & 1 & 1 & 1 & 0 & 1\\
0 & 0 & 1 & 2 & 0 & 1 & -2 & 1 & 2 & 1 & 0 & 0 & 1 & 0 & 2 & 1 & 0 & 1 & 0\\
1 & 1 & 2 & 1 & 2 & 0 & 1 & -2 & 1 & 1 & 2 & 2 & 0 & 0 & 2 & 1 & 1 & 0 & 1\\
2 & 1 & 1 & 0 & 1 & 2 & 2 & 1 & -2 & 2 & 1 & 2 & 1 & 1 & 1 & 1 & 0 & 1 & 1\\
1 & 1 & 1 & 0 & 2 & 0 & 1 & 1 & 2 & -2 & 0 & 0 & 2 & 2 & 1 & 1 & 1 & 0 & 0\\
0 & 0 & 0 & 1 & 1 & 1 & 0 & 2 & 1 & 0 & -2 & 0 & 2 & 2 & 1 & 1 & 1 & 1 & 0\\
0 & 0 & 0 & 1 & 1 & 1 & 0 & 2 & 2 & 0 & 0 & -2 & 1 & 1 & 0 & 2 & 1 & 1 & 0\\
0 & 1 & 1 & 1 & 1 & 1 & 1 & 0 & 1 & 2 & 2 & 1 & -2 & 0 & 2 & 0 & 1 & 1 & 1\\
1 & 0 & 2 & 2 & 0 & 0 & 0 & 0 & 1 & 2 & 2 & 1 & 0 & -2 & 2 & 2 & 0 & 1 & 1\\
2 & 1 & 1 & 1 & 1 & 1 & 2 & 2 & 1 & 1 & 1 & 0 & 2 & 2 & -2 & 2 & 1 & 0 & 1\\
0 & 1 & 1 & 1 & 1 & 1 & 1 & 1 & 1 & 1 & 1 & 2 & 0 & 2 & 2 & -2 & 1 & 0 & 0\\
1 & 1 & 1 & 1 & 0 & 1 & 0 & 1 & 0 & 1 & 1 & 1 & 1 & 0 & 1 & 1 & -2 & 1 & 0\\
1 & 1 & 1 & 0 & 1 & 0 & 1 & 0 & 1 & 0 & 1 & 1 & 1 & 1 & 0 & 0 & 1 & -2 & 1\\
0 & 0 & 0 & 1 & 0 & 1 & 0 & 1 & 1 & 0 & 0 & 0 & 1 & 1 & 1 & 0 & 0 & 1 & -2
\end{array}\right)
\]

\vfill{}

\newpage{}

For a very general member of the family $\mathcal{X}_{C}$, a full
rank minor of minimal discriminant is

\[
M_{C}=\left(\begin{array}{rrrrrrrrrrrrrrrrrrr}
-2 & 0 & 0 & 2 & 2 & 0 & 0 & 1 & 2 & 1 & 0 & 0 & 0 & 1 & 2 & 0 & 1 & 0 & 1\\
0 & -2 & 2 & 2 & 0 & 0 & 0 & 1 & 1 & 1 & 0 & 0 & 1 & 0 & 1 & 1 & 1 & 1 & 1\\
0 & 2 & -2 & 0 & 2 & 2 & 1 & 2 & 1 & 1 & 0 & 0 & 1 & 2 & 1 & 1 & 1 & 0 & 1\\
2 & 2 & 0 & -2 & 2 & 2 & 2 & 1 & 0 & 0 & 1 & 1 & 1 & 2 & 1 & 1 & 1 & 0 & 0\\
2 & 0 & 2 & 2 & -2 & 2 & 0 & 2 & 1 & 2 & 1 & 1 & 1 & 0 & 1 & 1 & 0 & 1 & 1\\
0 & 0 & 2 & 2 & 2 & -2 & 1 & 0 & 2 & 0 & 1 & 1 & 1 & 0 & 1 & 1 & 1 & 1 & 0\\
0 & 0 & 1 & 2 & 0 & 1 & -2 & 1 & 2 & 1 & 0 & 0 & 1 & 0 & 2 & 1 & 0 & 0 & 1\\
1 & 1 & 2 & 1 & 2 & 0 & 1 & -2 & 1 & 1 & 2 & 2 & 0 & 0 & 2 & 1 & 1 & 0 & 0\\
2 & 1 & 1 & 0 & 1 & 2 & 2 & 1 & -2 & 2 & 1 & 2 & 1 & 1 & 1 & 1 & 0 & 1 & 1\\
1 & 1 & 1 & 0 & 2 & 0 & 1 & 1 & 2 & -2 & 0 & 0 & 2 & 2 & 1 & 1 & 1 & 0 & 0\\
0 & 0 & 0 & 1 & 1 & 1 & 0 & 2 & 1 & 0 & -2 & 0 & 2 & 2 & 1 & 1 & 1 & 0 & 1\\
0 & 0 & 0 & 1 & 1 & 1 & 0 & 2 & 2 & 0 & 0 & -2 & 1 & 1 & 0 & 2 & 1 & 0 & 1\\
0 & 1 & 1 & 1 & 1 & 1 & 1 & 0 & 1 & 2 & 2 & 1 & -2 & 0 & 2 & 0 & 1 & 0 & 1\\
1 & 0 & 2 & 2 & 0 & 0 & 0 & 0 & 1 & 2 & 2 & 1 & 0 & -2 & 2 & 2 & 0 & 1 & 1\\
2 & 1 & 1 & 1 & 1 & 1 & 2 & 2 & 1 & 1 & 1 & 0 & 2 & 2 & -2 & 2 & 1 & 1 & 0\\
0 & 1 & 1 & 1 & 1 & 1 & 1 & 1 & 1 & 1 & 1 & 2 & 0 & 2 & 2 & -2 & 1 & 0 & 0\\
1 & 1 & 1 & 1 & 0 & 1 & 0 & 1 & 0 & 1 & 1 & 1 & 1 & 0 & 1 & 1 & -2 & 1 & 1\\
0 & 1 & 0 & 0 & 1 & 1 & 0 & 0 & 1 & 0 & 0 & 0 & 0 & 1 & 1 & 0 & 1 & -2 & 0\\
1 & 1 & 1 & 0 & 1 & 0 & 1 & 0 & 1 & 0 & 1 & 1 & 1 & 1 & 0 & 0 & 1 & 0 & -2
\end{array}\right)
\]

\vfill{}

For the surface $Y$, a full rank minor of minimal discriminant is

\selectlanguage{english}%
\[
M_{Y}=\left(\begin{array}{rrrrrrrrrrrrrrrrrrrr}
-2 & 0 & 0 & 2 & 2 & 0 & 0 & 1 & 2 & 1 & 0 & 0 & 0 & 1 & 2 & 0 & 1 & 2 & 1 & 0\\
0 & -2 & 2 & 2 & 0 & 0 & 0 & 1 & 1 & 1 & 0 & 0 & 1 & 0 & 1 & 1 & 1 & 2 & 1 & 0\\
0 & 2 & -2 & 0 & 2 & 2 & 1 & 2 & 1 & 1 & 0 & 0 & 1 & 2 & 1 & 1 & 1 & 0 & 1 & 1\\
2 & 2 & 0 & -2 & 2 & 2 & 2 & 1 & 0 & 0 & 1 & 1 & 1 & 2 & 1 & 1 & 1 & 0 & 0 & 1\\
2 & 0 & 2 & 2 & -2 & 2 & 0 & 2 & 1 & 2 & 1 & 1 & 1 & 0 & 1 & 1 & 0 & 0 & 1 & 0\\
0 & 0 & 2 & 2 & 2 & -2 & 1 & 0 & 2 & 0 & 1 & 1 & 1 & 0 & 1 & 1 & 1 & 2 & 0 & 0\\
0 & 0 & 1 & 2 & 0 & 1 & -2 & 1 & 2 & 1 & 0 & 0 & 1 & 0 & 2 & 1 & 0 & 1 & 1 & 0\\
1 & 1 & 2 & 1 & 2 & 0 & 1 & -2 & 1 & 1 & 2 & 2 & 0 & 0 & 2 & 1 & 1 & 1 & 0 & 0\\
2 & 1 & 1 & 0 & 1 & 2 & 2 & 1 & -2 & 2 & 1 & 2 & 1 & 1 & 1 & 1 & 0 & 0 & 0 & 1\\
1 & 1 & 1 & 0 & 2 & 0 & 1 & 1 & 2 & -2 & 0 & 0 & 2 & 2 & 1 & 1 & 1 & 1 & 0 & 0\\
0 & 0 & 0 & 1 & 1 & 1 & 0 & 2 & 1 & 0 & -2 & 0 & 2 & 2 & 1 & 1 & 1 & 1 & 1 & 0\\
0 & 0 & 0 & 1 & 1 & 1 & 0 & 2 & 2 & 0 & 0 & -2 & 1 & 1 & 0 & 2 & 1 & 1 & 1 & 1\\
0 & 1 & 1 & 1 & 1 & 1 & 1 & 0 & 1 & 2 & 2 & 1 & -2 & 0 & 2 & 0 & 1 & 1 & 0 & 0\\
1 & 0 & 2 & 2 & 0 & 0 & 0 & 0 & 1 & 2 & 2 & 1 & 0 & -2 & 2 & 2 & 0 & 1 & 0 & 0\\
2 & 1 & 1 & 1 & 1 & 1 & 2 & 2 & 1 & 1 & 1 & 0 & 2 & 2 & -2 & 2 & 1 & 0 & 1 & 2\\
0 & 1 & 1 & 1 & 1 & 1 & 1 & 1 & 1 & 1 & 1 & 2 & 0 & 2 & 2 & -2 & 1 & 1 & 1 & 0\\
1 & 1 & 1 & 1 & 0 & 1 & 0 & 1 & 0 & 1 & 1 & 1 & 1 & 0 & 1 & 1 & -2 & 0 & 0 & 0\\
2 & 2 & 0 & 0 & 0 & 2 & 1 & 1 & 0 & 1 & 1 & 1 & 1 & 1 & 0 & 1 & 0 & -2 & 0 & 1\\
1 & 1 & 1 & 0 & 1 & 0 & 1 & 0 & 0 & 0 & 1 & 1 & 0 & 0 & 1 & 1 & 0 & 0 & -2 & 0\\
0 & 0 & 1 & 1 & 0 & 0 & 0 & 0 & 1 & 0 & 0 & 1 & 0 & 0 & 2 & 0 & 0 & 1 & 0 & -2
\end{array}\right)
\]

\selectlanguage{british}%
\bibliographystyle{amsalpha}
\bibliography{Articles_Ref}

\end{document}